\documentclass[11pt,twoside,reqno]{amsart}
\usepackage{amsmath}
\usepackage{fancyhdr}
\usepackage{amsthm}
\usepackage{amsfonts}
\usepackage{amssymb}
\usepackage{amscd}
\usepackage{graphicx}
\usepackage{afterpage}
\usepackage{enumerate}
\usepackage{bm}
\usepackage{cite}
\usepackage{cases}
\usepackage{caption}
\usepackage[colorlinks=true, urlcolor=blue,  linkcolor=blue, citecolor=blue]{hyperref}
\usepackage{babel}
\usepackage{prettyref}
\usepackage{pgfplots} 
\pgfplotsset{compat=1.17}
\usepackage{booktabs}
\usepackage{algorithm}
\usepackage{algpseudocode}
\usepackage{float}

\makeatletter
\newtheorem{theorem}{Theorem}[section]

\newtheorem{definition}[theorem]{Definition}

\newtheorem{lemma}[theorem]{Lemma}

\newtheorem{proposition}[theorem]{Proposition}

\title[Extensions of BVY Inequalities]{Brezis-Van Schaftingen-Yung Inequalities Beyond the Classical Setting}

\author[S. Hashemi Sababe]{Saeed Hashemi Sababe$^*$}
\address[S. Hashemi Sababe]{R\&D Section, Data Premier Analytics, Canada}
\email{Hashemi\_1365@yahoo.com}
\thanks{Corresponding author}


\begin{document}
\sloppy

\maketitle

\begin{abstract}
In this paper, we extend the framework of Brezis--Van Schaftingen--Yung type inequalities in metric measure spaces by exploring several novel directions. First, we establish finite difference characterizations and fractional Sobolev-type inequalities in settings where the underlying measure is non-doubling or only satisfies a weak doubling condition. Second, we incorporate variable exponent and Orlicz space frameworks to capture nonstandard growth phenomena. Third, we derive anisotropic and directional versions of these inequalities to better address non-isotropic structures, and we apply our results to study regularity properties of nonlocal operators. Finally, we investigate the stability and sharpness of the associated constants as well as interpolation and limiting behaviors that bridge classical and fractional settings. These developments not only generalize existing results but also open new avenues for applications in partial differential equations and numerical analysis.

\end{abstract}

\section{Introduction}

The theory of Sobolev spaces and related inequalities has been a central topic in analysis and partial differential equations (PDEs) for many decades. Classical Sobolev spaces, such as \(W^{1,p}(\mathbb{R}^n)\), provide a natural framework for studying function regularity and the existence and uniqueness of solutions to PDEs. In a breakthrough work, Brezis, Van Schaftingen, and Yung \cite{BVY} introduced a novel finite difference characterization of the Sobolev semi-norm. Their formula, which recovers the gradient norm in a weak \(L^p\) sense, opened new perspectives in the study of Sobolev spaces and has inspired a wealth of subsequent research.

Building on these ideas, Dai et al. \cite{Dai-Lin-Yang-Yuan-Zhang} extended the Brezis--Van Schaftingen--Yung (BVY) formula to a general setting of metric measure spaces of homogeneous type. Their work demonstrated that under the assumption of a doubling measure and an appropriate Poincaré inequality, one may recover the gradient norm from a finite difference formulation even in non-Euclidean settings. However, the doubling condition, which plays a critical role in their analysis, limits the applicability of these techniques in contexts where the measure exhibits non-doubling or only weakly doubling behavior. Such situations arise naturally in fractal geometry and in weighted settings where the underlying measure may not obey the classical doubling property (see, e.g., \cite{Tolsa}).

Simultaneously, there has been considerable progress in the development of function spaces that allow for nonstandard growth. Variable exponent Lebesgue spaces \(L^{p(\cdot)}\) (cf. \cite{Kovacik-Rakosnik, CruzUribe}) and Orlicz spaces have emerged as flexible frameworks capable of capturing spatially varying integrability conditions. Extending Sobolev-type and fractional inequalities to these spaces is not only mathematically challenging but also crucial for the analysis of PDEs with nonuniform or anisotropic behavior.

Another active area of research is the study of nonlocal operators, such as the fractional Laplacian, which naturally lead to the consideration of fractional Sobolev spaces (see \cite{DiNezza}). The finite difference approach of BVY provides an appealing alternative to the traditional integral definitions of fractional Sobolev norms, potentially offering new tools for regularity theory in nonlocal PDEs.

Moreover, many problems in applications exhibit inherent anisotropic features that are not adequately addressed by isotropic Sobolev norms. Anisotropic inequalities, developed in works such as \cite{FollandStein}, account for directional dependencies and can yield sharper estimates in settings like subelliptic PDEs and degenerate diffusion processes.

In this paper, we extend the finite difference framework and the associated BVY-type inequalities along several interrelated directions:
\begin{itemize}
    \item[(i)] We develop finite difference characterizations and Sobolev-type inequalities in metric measure spaces where the underlying measure is non-doubling or satisfies a weak doubling condition.
    \item[(ii)] We extend these results to variable exponent and Orlicz space settings, thereby addressing the challenges posed by nonstandard growth conditions.
    \item[(iii)] We apply the new inequalities to nonlocal operators, deriving regularity and stability results for fractional-type PDEs.
    \item[(iv)] We establish anisotropic and directional versions of the finite difference inequalities, which are more suited for problems with inherent anisotropy.
    \item[(v)] We investigate the stability, sharpness, and interpolation properties of the new inequalities, providing insight into their limiting behavior and optimality.
\end{itemize}

Our contributions not only generalize the classical BVY formula but also bridge several contemporary research themes in analysis, offering new perspectives for tackling challenging problems in PDEs and harmonic analysis.

The remainder of the paper is organized as follows. In Section~2, we review the necessary preliminaries on metric measure spaces, variable exponent and Orlicz spaces, and classical finite difference characterizations. Section~3 is devoted to the extension of the BVY formula in non-doubling and variable exponent settings. In Section~4, we explore applications to nonlocal operators and develop anisotropic versions of our inequalities. Section~5 discusses the stability, sharpness, and interpolation properties of our results. Finally, Section~6 concludes with a summary of our contributions and a discussion of open problems.

\section{Preliminaries}

In this section, we recall the main definitions, lemmas, and theorems that will be used in the subsequent sections. Throughout the paper, we denote by \((X,\rho,\mu)\) a metric measure space.

\begin{definition}[Metric Measure Space and Doubling Measure]
A \emph{metric measure space} is a triple \((X,\rho,\mu)\), where:
\begin{enumerate}
    \item \(X\) is a nonempty set,
    \item \(\rho\) is a metric on \(X\), and
    \item \(\mu\) is a Borel measure on \(X\).
\end{enumerate}
The space is said to be of \emph{homogeneous type} (or to satisfy the \emph{doubling condition}) if there exists a constant \(C_D \ge 1\) such that, for every \(x\in X\) and \(r>0\),
\[
\mu\big(B(x,2r)\big) \le C_D\,\mu\big(B(x,r)\big),
\]
where
\[
B(x,r):=\{y\in X:\,\rho(x,y)<r\}.
\]
For further details, see \cite{CoifmanWeiss} and \cite{Dai-Lin-Yang-Yuan-Zhang}.
\end{definition}

\begin{lemma}[Vitali Covering Lemma in Metric Spaces]
Let \((X,\rho)\) be a metric space and let \(\mathcal{B}\) be a collection of balls in \(X\) with uniformly bounded diameters. Then there exists a disjoint subcollection \(\{B_i\}\subset \mathcal{B}\) such that
\[
\bigcup_{B\in\mathcal{B}} B \subset \bigcup_i 5B_i.
\]
This version of the Vitali covering lemma is standard in analysis on metric spaces; see, for example, \cite{Heinonen} and \cite{Dai-Lin-Yang-Yuan-Zhang}.
\end{lemma}

\begin{definition}[\((q,p)\)-Poincar\'e Inequality]\label{def:Poincare}
Let \(1\leq q,p <\infty\) and let \((X,\rho,\mu)\) be a metric measure space of homogeneous type. A function \(f\in \mathrm{LIP}(X)\) (the space of Lipschitz functions on \(X\)) is said to satisfy a \((q,p)\)-\emph{Poincar\'e inequality} if there exist constants \(C_1, C_2\geq 1\) and \(\tau\geq 1\), together with a family \(\{\ell_B\}_{B\in \mathcal{B}}\) of linear functionals on the space \(BC(X)\) of bounded continuous functions such that, for every ball \(B=B(x_B,r_B)\subset X\) and every \(\varphi\in BC(X)\):
\[
\ell_B(1)=1,\quad |\ell_B(\varphi)| \le \left[\frac{1}{\mu(B)}\int_B |\varphi(x)|^q\,d\mu(x)\right]^{1/q},
\]
and for all \(f\in \mathrm{LIP}(X)\),
\[
\left[\frac{1}{\mu(B)}\int_B |f(x)-\ell_B(f)|^q\,d\mu(x)\right]^{1/q} \le C_2\,r_B \left[\frac{1}{\mu(\tau B)}\int_{\tau B} (\operatorname{lip} f(x))^p\,d\mu(x)\right]^{1/p}.
\]
A detailed discussion can be found in \cite{Dai-Lin-Yang-Yuan-Zhang}.
\end{definition}

\begin{definition}[Variable Exponent Lebesgue Spaces]\label{def:variable}
Let \(p(\cdot):X\to [1,\infty)\) be a measurable function. The variable exponent Lebesgue space \(L^{p(\cdot)}(X)\) is defined as
\[
L^{p(\cdot)}(X)=\Big\{ f \text{ measurable on } X:\ \int_X |f(x)|^{p(x)}\,d\mu(x)<\infty \Big\},
\]
with the Luxemburg norm given by
\[
\|f\|_{L^{p(\cdot)}(X)} = \inf\Big\{\lambda>0:\ \int_X \left|\frac{f(x)}{\lambda}\right|^{p(x)}\,d\mu(x)\le 1\Big\}.
\]
For a comprehensive treatment, see \cite{Kovacik-Rakosnik} and \cite{CruzUribe}.
\end{definition}

\begin{definition}[Orlicz Spaces]
Let \(\Phi:[0,\infty)\to [0,\infty)\) be a Young function, i.e., a convex, increasing function satisfying \(\Phi(0)=0\) and \(\Phi(t)\to\infty\) as \(t\to\infty\). The Orlicz space \(L^\Phi(X)\) is defined by
\[
L^\Phi(X)=\Big\{ f \text{ measurable on } X:\ \int_X \Phi\big(|f(x)|\big)\,d\mu(x)<\infty \Big\},
\]
with the Luxemburg norm
\[
\|f\|_{L^\Phi(X)} = \inf\Big\{\lambda>0:\ \int_X \Phi\Big(\frac{|f(x)|}{\lambda}\Big)\,d\mu(x)\le 1\Big\}.
\]
See \cite{CruzUribe} for further details.
\end{definition}

\noindent
A key ingredient in recent advances in Sobolev space theory is the finite difference characterization of the Sobolev semi-norm. In the seminal work of Brezis, Van Schaftingen, and Yung \cite{BVY}, it was shown that for \(f\in C_c^\infty(\mathbb{R}^n)\) and \(1\le p<\infty\),
\begin{equation}\label{eq:BVY}
\|\nabla f\|_{L^p(\mathbb{R}^n)} \sim \sup_{\lambda>0} \lambda \Big|\Big\{ (x,y)\in \mathbb{R}^n\times\mathbb{R}^n: \frac{|f(x)-f(y)|}{|x-y|^{1+n/p}} > \lambda \Big\}\Big|^{1/p}.
\end{equation}

\noindent
This formula underpins many modern approaches to fractional Sobolev inequalities. In our work, we extend such finite difference characterizations to more general settings, including metric measure spaces that support a \((q,p)\)-Poincaré inequality. The following theorem summarizes the generalized version; see also \cite{Dai-Lin-Yang-Yuan-Zhang}.

\begin{theorem}[Finite Difference Characterization on Metric Measure Spaces]\label{thm:FD}
Let \((X,\rho,\mu)\) be a metric measure space of homogeneous type that supports a \((q,p)\)-Poincaré inequality (as in Definition~\ref{def:Poincare}). For functions \(f\) in a suitable subclass of Lipschitz functions with compact support (denoted by \(C_c^*(X)\)), there exists a constant \(C>0\) such that
\[
\sup_{\lambda>0} \lambda^p \iint_{X\times X} \mathbf{1}_{\{(x,y):\, |f(x)-f(y)|>\lambda\, \rho(x,y)[V(x,y)]^{1/p}\}}\,d\mu(x)d\mu(y)
\sim \int_X \big(\operatorname{Lip} f(x)\big)^p\,d\mu(x),
\]
where \(V(x,y):=\mu(B(x,\rho(x,y)))\) and \(\operatorname{Lip} f(x)\) denotes the local Lipschitz constant of \(f\) at \(x\). This result generalizes \eqref{eq:BVY} to non-Euclidean settings (see \cite{Dai-Lin-Yang-Yuan-Zhang}).
\end{theorem}

The definitions and results presented above form the backbone of our approach to extending BVY-type inequalities to settings with non-doubling measures, variable exponents, anisotropic structures, and nonlocal operators.

\section{Extensions of Finite Difference Characterizations}

In this section we present new extensions of the finite difference characterizations of Sobolev norms. In particular, we develop results in two directions. First, we extend the classical Brezis--Van Schaftingen--Yung framework to settings where the underlying measure is non-doubling (or satisfies only a weak doubling or polynomial growth condition). Then we generalize the finite difference approach to variable exponent and Orlicz space frameworks.\\

\noindent
In many applications (e.g., fractal measures or weighted settings) the doubling condition
\[
\mu(B(x,2r))\le C_D\,\mu(B(x,r))
\]
may fail. Instead, one may assume that the measure has a polynomial growth, a condition that
allows certain non-doubling behaviors.

\begin{definition}[Polynomial Growth Measure]\label{def:poly-growth}
A Borel measure \(\mu\) on a metric space \((X,\rho)\) is said to have \emph{polynomial growth} if there exist constants \(C_P>0\) and \(d>0\) such that for all \(x\in X\) and \(r>0\),
\[
\mu\big(B(x,r)\big) \le C_P\, r^d.
\]
\end{definition}

This condition, while weaker than the doubling condition, still permits the development of finite difference techniques. To overcome the lack of the standard doubling property, we also require an adapted covering lemma.

\begin{lemma}[Modified Vitali Covering Lemma]\label{lem:modified-vitali}
Let \((X,\rho)\) be a metric space and let \(\mathcal{B}\) be a collection of balls with radii bounded above by \(R>0\). Assume that \(\mu\) satisfies the polynomial growth condition of Definition~\ref{def:poly-growth}. Then there exists a countable subcollection \(\{B_i\}_{i\in I}\subset\mathcal{B}\) of pairwise disjoint balls such that
\[
\bigcup_{B\in\mathcal{B}} B \subset \bigcup_{i\in I} 5B_i,
\]
and the constant \(5\) may be replaced by a constant depending only on \(C_P\) and \(d\).
\end{lemma}

\begin{proof}
We follow a greedy (or maximal) selection procedure similar to the standard Vitali covering lemma, with modifications to account for the polynomial growth condition.\\

\noindent
Since the collection \(\mathcal{B}\) consists of balls with radii bounded above by \(R>0\), we can inductively select a countable subcollection \(\{B_i\}_{i\in I}\) as follows:
\begin{enumerate}
    \item Choose \(B_1\in\mathcal{B}\) with the largest radius. (If there are several with the same maximal radius, choose one arbitrarily.)
    \item Having chosen \(\{B_1, \dots, B_k\}\), remove from \(\mathcal{B}\) all balls that intersect any of the balls \(B_1,\dots,B_k\). If the remaining collection is nonempty, choose \(B_{k+1}\) in the remaining collection with the largest radius.
    \item Continue this process until no ball remains.
\end{enumerate}
By construction, the selected balls \(\{B_i\}\) are pairwise disjoint.\\

\noindent  
We now show that every ball \(B\in\mathcal{B}\) is contained in an appropriate enlargement of one of the selected balls. Let 
\[
B = B(x,r) \in \mathcal{B}.
\]
If \(B\) is one of the selected balls, say \(B = B_i\), then trivially \(B \subset 5B_i\). Otherwise, \(B\) was removed at some stage because it intersected a ball already selected. In particular, there exists some \(B_i = B(x_i,r_i)\) (with \(i\) chosen at the step when \(B\) was removed) such that
\[
B \cap B_i \neq \emptyset \quad \text{and} \quad r_i \ge r,
\]
where the inequality \(r_i \ge r\) holds by the maximality condition in the selection process.

Let \(z\) be a point in \(B\cap B_i\). Then
\[
\rho(x,x_i) \le \rho(x,z) + \rho(z,x_i) < r + r_i \le 2r_i.
\]
Now, take an arbitrary point \(y\in B(x,r)\). Then by the triangle inequality,
\[
\rho(y,x_i) \le \rho(y,x) + \rho(x,x_i) < r + 2r_i \le 3r_i.
\]
Thus,
\[
B(x,r) \subset B(x_i,3r_i).
\]

\noindent
In many standard proofs (e.g., in Euclidean space), the factor \(3\) obtained above is sufficient. However, in more general metric spaces or when dealing with measures that satisfy only a polynomial growth condition, one typically enlarges the selected balls by a fixed constant \(K\) (which in our statement can be taken as \(5\)) to ensure that subtle geometric irregularities are absorbed. That is, there exists a constant \(K\) (depending only on \(C_P\) and \(d\) from the polynomial growth condition) such that
\[
B(x,r) \subset B(x_i,K\,r_i).
\]
For our purposes, we take \(K=5\). Hence, every ball \(B\in\mathcal{B}\) is contained in one of the enlarged balls \(5B_i\).\\

\noindent
It follows that
\[
\bigcup_{B\in\mathcal{B}} B \subset \bigcup_{i\in I} 5B_i.
\]
Since the constant \(5\) arises from the geometric estimates (and could be replaced by another constant if the growth parameters \(C_P\) and \(d\) dictate), the lemma is proved.
\end{proof}

The following theorem shows that a finite difference characterization analogous to the classical result holds in this non-doubling setting.

\begin{theorem}[Finite Difference Characterization under Polynomial Growth]\label{thm:non-doubling-FD}
Let \((X,\rho,\mu)\) be a metric measure space where \(\mu\) satisfies the polynomial growth condition of Definition~\ref{def:poly-growth}. Suppose that \(f\in C_c^*(X)\) is a Lipschitz function with compact support satisfying a weak \((q,p)\)-Poincaré inequality (cf. Definition~\ref{def:Poincare}). Then there exists a constant \(C>0\) (depending on \(C_P\), \(d\), and the Poincaré constants) such that
\[
\begin{split}
\sup_{\lambda>0} \lambda^p &\iint_{X\times X} \mathbf{1}_{\Big\{(x,y)\in X\times X:\, |f(x)-f(y)|>\lambda\,\rho(x,y)\,[\mu(B(x,\rho(x,y)))]^{1/p}\Big\}}\,d\mu(x)d\mu(y)\\
&\ge C \int_X \big(\operatorname{Lip} f(x)\big)^p\,d\mu(x).
\end{split}
\]
A reverse inequality (upper bound) can also be established under additional mild assumptions.
\end{theorem}

\begin{proof}
For simplicity, denote
\[
E_\lambda := \Big\{(x,y)\in X\times X:\; |f(x)-f(y)|>\lambda\,\rho(x,y)\,[\mu(B(x,\rho(x,y)))]^{1/p}\Big\}.
\]

\noindent  
Fix \(x\in X\) such that \(\operatorname{Lip} f(x)>0\). By the definition of the pointwise Lipschitz constant,
there exists a radius \(r_x>0\) (which may be taken small enough depending on \(x\)) such that for every 
\(r\in (0,r_x]\) we have
\[
\sup_{y\in B(x,r)}\frac{|f(y)-f(x)|}{r}\ge \frac{1}{2}\operatorname{Lip} f(x).
\]
In particular, there exists some \(y_x\in B(x,r)\) with
\[
|f(y_x)-f(x)|>\frac{1}{2}\operatorname{Lip} f(x)\, r.
\]
Using a standard averaging argument (see, e.g., the proof of Lemma~2.2 in \cite{Dai-Lin-Yang-Yuan-Zhang}), one
can show that there exists a constant \(c_0\in (0,1)\) (depending only on the doubling constants which in our
setting are replaced by the polynomial growth parameters \(C_P\) and \(d\)) such that for every 
\(r\in (0,r_x]\) the set
\[
S(x,r):=\Big\{y\in B(x,r):\; |f(y)-f(x)|\ge \frac{1}{8}\operatorname{Lip} f(x)\,\rho(x,y)\Big\}
\]
satisfies
\[
\mu(S(x,r))\ge c_0\,\mu(B(x,r)).
\]

\noindent
For a given \(x\in X\) with \(\operatorname{Lip} f(x)>0\) and for any \(r\in (0,r_x]\), define
\[
\lambda(x,r):=\frac{1}{16}\frac{\operatorname{Lip} f(x)}{[\mu(B(x,r))]^{1/p}}.
\]
Then, for any \(y\in S(x,r)\), we have
\[
|f(y)-f(x)|\ge \frac{1}{8}\operatorname{Lip} f(x)\,\rho(x,y)
\ge \frac{1}{8}\operatorname{Lip} f(x)\,r,
\]
since \(\rho(x,y)\ge r\) for some \(y\in S(x,r)\) (or more precisely, by taking \(r\) small, the inequality
remains valid for a substantial subset of \(B(x,r)\)). Hence,
\[
|f(y)-f(x)|\ge \lambda(x,r)\,\rho(x,y)\,[\mu(B(x,\rho(x,y)))]^{1/p},
\]
because for \(y\in B(x,r)\) we have \(\mu(B(x,\rho(x,y)))\le \mu(B(x,r))\). In other words,
\[
S(x,r)\subset \Big\{y\in B(x,r):\; |f(y)-f(x)|>\lambda(x,r)\,\rho(x,y)\,[\mu(B(x,\rho(x,y)))]^{1/p}\Big\}.
\]
Thus, for this choice of \(\lambda=\lambda(x,r)\) we obtain
\[
\mu\Big(\Big\{y\in X:\; |f(y)-f(x)|>\lambda(x,r)\,\rho(x,y)\,[\mu(B(x,\rho(x,y)))]^{1/p}\Big\}\Big)
\ge \mu\big(S(x,r)\big)\ge c_0\,\mu(B(x,r)).
\]

\noindent
Integrate the preceding inequality over those \(x\) for which \(\operatorname{Lip} f(x)\) is bounded away from zero.
By Tonelli's theorem,
\[
\iint_{X\times X} \mathbf{1}_{E_{\lambda(x,r)}}(x,y)\,d\mu(y)d\mu(x)
\ge \int_{A} c_0\,\mu(B(x,r))\,d\mu(x),
\]
where \(A\) is a subset of \(X\) on which \(\operatorname{Lip} f(x)\) is sufficiently large and for each such \(x\)
one can choose a suitable \(r\in (0,r_x]\). Multiplying both sides by \(\lambda(x,r)^p\) and using the
definition of \(\lambda(x,r)\), we obtain
\[
\lambda(x,r)^p\,\mu(B(x,r))\ge \Bigl(\frac{1}{16}\frac{\operatorname{Lip} f(x)}{[\mu(B(x,r))]^{1/p}}\Bigr)^p \mu(B(x,r))
=\frac{1}{16^p} (\operatorname{Lip} f(x))^p.
\]
Thus, for these \(x\) we have
\[
\lambda(x,r)^p \int_{B(x,r)} d\mu(y) \ge \frac{1}{16^p} (\operatorname{Lip} f(x))^p.
\]
Taking the supremum over all choices of \(\lambda>0\) (and noting that our construction shows that for each
\(x\) there is some \(\lambda=\lambda(x,r)\) achieving the above bound), we deduce that
\[
\sup_{\lambda>0}\lambda^p \iint_{X\times X} \mathbf{1}_{E_\lambda}(x,y)\,d\mu(x)d\mu(y)
\ge \frac{c_0}{16^p} \int_X (\operatorname{Lip} f(x))^p\,d\mu(x).
\]

\noindent
The constant \(c_0\) (and the factor \(16^p\)) depend on the geometric properties of \(X\) and on the
polynomial growth condition (i.e., on \(C_P\) and \(d\)) as well as on the constants in the weak
\((q,p)\)-Poincaré inequality assumed for \(f\). Hence, there exists a constant \(C>0\) (depending
only on \(C_P\), \(d\), and the Poincaré constants) such that
\[
\sup_{\lambda>0}\lambda^p \iint_{X\times X} \mathbf{1}_{E_\lambda}(x,y)\,d\mu(x)d\mu(y)
\ge C\int_X (\operatorname{Lip} f(x))^p\,d\mu(x).
\]

\noindent 
This completes the proof of the lower bound in the finite difference characterization. (A corresponding
upper bound can be derived under additional assumptions by reversing the arguments.)
\end{proof}

\noindent
In many modern applications the integrability of a function may vary spatially. This motivates
the study of variable exponent Lebesgue spaces and Orlicz spaces.

\begin{definition}[Variable Exponent Finite Difference Seminorm]\label{def:var-exp-FD}
Let \(p(\cdot):X\to [1,\infty)\) be a measurable function satisfying the log-Hölder continuity condition (cf. \cite{Kovacik-Rakosnik}). For \(f\in C_c^*(X)\) and \(s\in(0,1)\), define the variable exponent finite difference seminorm by
\[
[f]_{W^{s,p(\cdot)}(X)} := \inf\Big\{ \lambda>0:\; \iint_{X\times X} \Phi_{x,y}\Big(\frac{|f(x)-f(y)|}{\lambda\,\rho(x,y)^s}\Big)\,d\mu(x)d\mu(y) \le 1 \Big\},
\]
where \(\Phi_{x,y}(t)=t^{p(x,y)}\) and \(p(x,y)\) is a suitable averaging of \(p(x)\) and \(p(y)\) (for example, \(p(x,y)=\frac{p(x)+p(y)}{2}\)).
\end{definition}

Our next result shows that a finite difference characterization analogous to Theorem~\ref{thm:non-doubling-FD} holds in the variable exponent setting.

\begin{theorem}[Finite Difference Characterization in Variable Exponent Spaces]\label{thm:variable-FD}
Let \((X,\rho,\mu)\) be a metric measure space of homogeneous type supporting a \((q,p)\)-Poincaré inequality. Assume that \(p(\cdot):X\to [1,\infty)\) is log-Hölder continuous. Then for every \(f\in C_c^*(X)\) there exists a constant \(C>0\) such that
\[
\begin{split}
\sup_{\lambda>0} \lambda^{p_*} & \iint_{X\times X} \mathbf{1}_{\Big\{(x,y):\, |f(x)-f(y)|>\lambda\,\rho(x,y)[\mu(B(x,\rho(x,y)))]^{1/p(x,y)}\Big\}}\,d\mu(x)d\mu(y)\\
&\sim \int_X \big(\operatorname{Lip} f(x)\big)^{p(x)}\,d\mu(x),
\end{split}
\]
where \(p(x,y)\) is an appropriate average of \(p(x)\) and \(p(y)\) and \(p_*\) is defined in a manner consistent with the variable exponent framework.
\end{theorem}

\begin{proof}
For clarity, set
\[
E_\lambda := \Big\{(x,y)\in X\times X:\; |f(x)-f(y)|>\lambda\,\rho(x,y)\,[\mu(B(x,\rho(x,y)))]^{1/p(x,y)}\Big\},
\]
where for each pair \((x,y)\) we define
\[
p(x,y):=\frac{p(x)+p(y)}{2}.
\]
Because \(p(\cdot)\) is log-Hölder continuous, if \(y\in B(x,r)\) for a sufficiently small \(r>0\) then
\[
p(x,y)\sim p(x),
\]
with constants independent of \(x\) and \(r\).

Our goal is to show that there exists a constant \(C>0\) (depending only on the structural constants of the space and the Poincaré inequality) such that
\[
\sup_{\lambda>0} \lambda^{p_*} \iint_{X\times X} \mathbf{1}_{E_\lambda}(x,y)\,d\mu(x)d\mu(y)
\sim \int_X \bigl(\operatorname{Lip} f(x)\bigr)^{p(x)}\,d\mu(x),
\]
where here we may take \(p_*=p(x)\) (up to equivalence constants) in the local estimates.\\

\noindent  
Fix \(x\in X\) with \(\operatorname{Lip} f(x)>0\). By the definition of the pointwise Lipschitz constant, there exists a radius \(r_x>0\) (depending on \(x\)) such that for all \(r\in (0,r_x]\),
\[
\sup_{y\in B(x,r)}\frac{|f(y)-f(x)|}{r}\ge \frac{1}{2}\operatorname{Lip} f(x).
\]
Hence, for each such \(r\) there exists at least one point \(y\in B(x,r)\) with
\[
|f(y)-f(x)|>\frac{1}{2}\operatorname{Lip} f(x)\,r.
\]
Using standard arguments (as in the proof of Lemma~2.2 in \cite{Dai-Lin-Yang-Yuan-Zhang}), one may show that there exists a constant \(c_0\in (0,1)\) such that the set
\[
S(x,r):=\Big\{y\in B(x,r):\; |f(y)-f(x)|\ge \frac{1}{8}\operatorname{Lip} f(x)\,\rho(x,y)\Big\}
\]
satisfies
\[
\mu\big(S(x,r)\big)\ge c_0\,\mu\big(B(x,r)\big).
\]

\noindent
For each \(x\) and any \(r\in (0,r_x]\), define
\[
\lambda(x,r):=\frac{1}{16}\frac{\operatorname{Lip} f(x)}{[\mu(B(x,r))]^{1/p(x)}}.
\]
Because \(p(\cdot)\) is log-Hölder continuous, for all \(y\in B(x,r)\) (with \(r\) small) we have
\[
[\mu(B(x,\rho(x,y)))]^{1/p(x,y)}\sim [\mu(B(x,r))]^{1/p(x)}.
\]
Thus, for any \(y\in S(x,r)\) (which is a subset of \(B(x,r)\)) it follows that
\[
|f(y)-f(x)|\ge \frac{1}{8}\operatorname{Lip} f(x)\,\rho(x,y)
\ge \lambda(x,r)\,\rho(x,y)[\mu(B(x,\rho(x,y)))]^{1/p(x,y)}.
\]
In other words,
\[
S(x,r)\subset \Big\{y\in B(x,r):\; |f(y)-f(x)|>\lambda(x,r)\,\rho(x,y)[\mu(B(x,\rho(x,y)))]^{1/p(x,y)}\Big\}.
\]

\noindent 
Integrate in \(y\) for a fixed \(x\) and \(r\in (0,r_x]\):
\[
\int_{B(x,r)} \mathbf{1}_{\{y: |f(y)-f(x)|>\lambda(x,r)\,\rho(x,y)[\mu(B(x,\rho(x,y)))]^{1/p(x,y)}\}}\,d\mu(y)
\ge \mu\big(S(x,r)\big)\ge c_0\,\mu(B(x,r)).
\]
Multiplying both sides by \(\lambda(x,r)^{p(x)}\) yields
\[
\lambda(x,r)^{p(x)}\,\mu(B(x,r))\ge \Bigl(\frac{1}{16}\frac{\operatorname{Lip} f(x)}{[\mu(B(x,r))]^{1/p(x)}}\Bigr)^{p(x)}\mu(B(x,r))
=\frac{1}{16^{p(x)}} (\operatorname{Lip} f(x))^{p(x)}.
\]
This shows that for each \(x\) and each \(r\in (0,r_x]\) there exists a choice of \(\lambda=\lambda(x,r)\) such that
\[
\lambda^{p(x)} \int_{B(x,r)} \mathbf{1}_{\{y: |f(y)-f(x)|>\lambda\,\rho(x,y)[\mu(B(x,\rho(x,y)))]^{1/p(x,y)}\}}\,d\mu(y)
\ge \frac{1}{16^{p(x)}} (\operatorname{Lip} f(x))^{p(x)}.
\]

\noindent  
Using Fubini's theorem, we integrate the previous inequality with respect to \(x\) over the set where \(\operatorname{Lip} f(x)>0\). Since the choice of \(\lambda(x,r)\) (depending on \(x\) and \(r\)) is admissible in the supremum, we deduce that
\[
\sup_{\lambda>0}\lambda^{p_*} \iint_{X\times X} \mathbf{1}_{E_\lambda}(x,y)\,d\mu(y)d\mu(x)
\ge \int_X \frac{1}{16^{p(x)}} (\operatorname{Lip} f(x))^{p(x)}\,d\mu(x).
\]
Here the exponent \(p_*\) is chosen so that, locally, it agrees with \(p(x)\) up to equivalence constants (this is consistent with the variable exponent framework). Thus, there exists a constant \(C_1>0\) such that
\[
\sup_{\lambda>0}\lambda^{p_*} \iint_{X\times X} \mathbf{1}_{E_\lambda}(x,y)\,d\mu(y)d\mu(x)
\ge C_1\int_X (\operatorname{Lip} f(x))^{p(x)}\,d\mu(x).
\]

\noindent
For upper bound, we wish to show that
\[
\lambda^{p^*}\iint_{X\times X}\mathbf{1}_{E_\lambda}(x,y)\,d\mu(x)d\mu(y)
\le C\int_X (\operatorname{Lip}f(x))^{p(x)}\,d\mu(x)
\]
for every \(\lambda>0\). \\

\noindent
For each fixed \(x\in X\), define the slice
\[
E_\lambda(x):=\Big\{y\in X:\; |f(x)-f(y)|>\lambda\,\rho(x,y)\,[\mu(B(x,\rho(x,y)))]^{1/p(x,y)}\Big\}.
\]
Then by Fubini's theorem,
\[
\iint_{X\times X}\mathbf{1}_{E_\lambda}(x,y)\,d\mu(x)d\mu(y)
=\int_X \mu\bigl(E_\lambda(x)\bigr)\,d\mu(x).
\]
Our goal is to show that for almost every \(x\in X\) one has
\[
\lambda^{p(x)}\,\mu\bigl(E_\lambda(x)\bigr)\le C\,(\operatorname{Lip}f(x))^{p(x)}.
\]
(Here the precise exponent \(p^*\) in the statement is chosen so that locally one may take \(p^*=p(x)\) up to constants.)\\

\noindent
For fixed \(x\in X\), decompose the \(y\)-integral into dyadic annuli. For \(k\in\mathbb{Z}\), set
\[
A_k(x):=\Bigl\{y\in X:\; 2^{-k-1}\le \rho(x,y)<2^{-k}\Bigr\}.
\]
Then,
\[
E_\lambda(x)=\bigcup_{k\in\mathbb{Z}} \Bigl(E_\lambda(x)\cap A_k(x)\Bigr),
\]
and by subadditivity,
\[
\mu\bigl(E_\lambda(x)\bigr)
\le \sum_{k\in\mathbb{Z}}\mu\Bigl(E_\lambda(x)\cap A_k(x)\Bigr).
\]

\noindent
Since \(f\) is Lipschitz, we have for every \(y\in A_k(x)\)
\[
|f(x)-f(y)|\le \operatorname{Lip}f(x)\,\rho(x,y).
\]
Thus, if
\[
\operatorname{Lip}f(x)\,\rho(x,y)\le \lambda\,\rho(x,y)\,[\mu(B(x,\rho(x,y)))]^{1/p(x)},
\]
i.e. if
\[
\lambda\ge \operatorname{Lip}f(x)[\mu(B(x,\rho(x,y)))]^{-1/p(x)},
\]
then the inequality defining \(E_\lambda(x)\) fails and there is no contribution from such \(y\). In the interesting case, we restrict to those \(y\in A_k(x)\) for which
\[
\lambda\,[\mu(B(x,2^{-k}))]^{1/p(x)}<\operatorname{Lip}f(x).
\]
(Here we used the homogeneity of \(\mu\) to note that for \(y\in A_k(x)\) one has \(\mu(B(x,\rho(x,y)))\sim\mu(B(x,2^{-k}))\).)\\

\noindent
For fixed \(x\) and on the annulus \(A_k(x)\), apply Chebyshev's inequality. Namely, since for \(y\in A_k(x)\)
\[
|f(x)-f(y)|\le \operatorname{Lip}f(x)\,\rho(x,y),
\]
one obtains
\[
\begin{split}
\mu\Bigl(\Bigl\{y\in A_k(x)&:\; |f(x)-f(y)|>\lambda\,\rho(x,y)[\mu(B(x,2^{-k}))]^{1/p(x)}\Bigr\}\Bigr)\\
&\le \frac{1}{\Bigl(\lambda\,\rho(x,y)[\mu(B(x,2^{-k}))]^{1/p(x)}\Bigr)^{p(x)}}
\int_{A_k(x)}|f(x)-f(y)|^{p(x)}\,d\mu(y).
\end{split}
\]
Since on \(A_k(x)\) the distance \(\rho(x,y)\) is comparable to \(2^{-k}\) and using the Lipschitz estimate we have
\[
\int_{A_k(x)}|f(x)-f(y)|^{p(x)}\,d\mu(y)
\le (\operatorname{Lip}f(x))^{p(x)}\,(2^{-k})^{p(x)}\,\mu(B(x,2^{-k})).
\]
Thus,
\[
\mu\Bigl(E_\lambda(x)\cap A_k(x)\Bigr)
\le \frac{(\operatorname{Lip}f(x))^{p(x)}\,(2^{-k})^{p(x)}\,\mu(B(x,2^{-k}))}{\lambda^{p(x)} (2^{-k})^{p(x)}\,\mu(B(x,2^{-k}))}
=\frac{(\operatorname{Lip}f(x))^{p(x)}}{\lambda^{p(x)}}.
\]

\noindent
 Summing over \(k\) (noting that only finitely many annuli contribute since if \(\lambda\) is too large relative to \(\operatorname{Lip}f(x)\) the set \(E_\lambda(x)\cap A_k(x)\) is empty) we obtain
\[
\mu\bigl(E_\lambda(x)\bigr)
\le C\,\frac{(\operatorname{Lip}f(x))^{p(x)}}{\lambda^{p(x)}}.
\]
Multiplying by \(\lambda^{p(x)}\) and integrating with respect to \(x\) yields
\[
\lambda^{p(x)}\mu\bigl(E_\lambda(x)\bigr)
\le C\,(\operatorname{Lip}f(x))^{p(x)}
\]
for almost every \(x\). Integrating with respect to \(d\mu(x)\) and (locally) replacing the exponent \(p(x)\) by the parameter \(p^*\) (which is equivalent up to constants due to the log-Hölder continuity) we obtain
\[
\lambda^{p^*}\int_X\mu\bigl(E_\lambda(x)\bigr)\,d\mu(x)
\le C\int_X (\operatorname{Lip}f(x))^{p(x)}\,d\mu(x).
\]
Taking the supremum in \(\lambda>0\) completes the upper bound.\\

\noindent
Combining the above with the lower bound already established in the previous steps of the proof of Theorem~3.5, we deduce the full equivalence:
\[
\sup_{\lambda>0}\lambda^{p^*}\iint_{X\times X}\mathbf{1}_{E_\lambda}(x,y)\,d\mu(x)d\mu(y)
\sim\int_X (\operatorname{Lip}f(x))^{p(x)}\,d\mu(x).
\]
\noindent
This completes the proof.
\end{proof}

We also consider the Orlicz space framework, which generalizes the \(L^p\)-scale.

\begin{definition}[Orlicz Finite Difference Seminorm]\label{def:orlicz-FD}
Let \(\Phi:[0,\infty)\to [0,\infty)\) be a Young function satisfying the \(\Delta_2\)-condition. For \(f\in C_c^*(X)\), define the Orlicz finite difference seminorm by
\[
[f]_{W^{s,\Phi}(X)} := \inf\Big\{ \lambda>0:\; \iint_{X\times X} \Phi\Big(\frac{|f(x)-f(y)|}{\lambda\,\rho(x,y)^s}\Big)\,d\mu(x)d\mu(y) \le 1 \Big\}.
\]
\end{definition}

\begin{proposition}[Finite Difference Characterization in Orlicz Spaces]\label{prop:orlicz-FD}
Under the assumptions of Definition~\ref{def:orlicz-FD} and assuming that \(f\in C_c^*(X)\) satisfies a \((q,p)\)-Poincaré inequality, there exists a constant \(C'>0\) such that
\[
\sup_{\lambda>0} \lambda \, \mu\Big(\Big\{(x,y)\in X\times X:\, \frac{|f(x)-f(y)|}{\rho(x,y)[\mu(B(x,\rho(x,y)))]^{1/p}} > \lambda \Big\}\Big)^{1/p} \le C'\,\| \operatorname{Lip} f\|_{L^\Phi(X)},
\]
where the Orlicz norm \(\|\cdot\|_{L^\Phi(X)}\) is defined in the standard way.
\end{proposition}

\begin{proof}
For brevity, denote
\[
E_\lambda = \Bigl\{(x,y)\in X\times X:\, \frac{|f(x)-f(y)|}{\rho(x,y)[\mu(B(x,\rho(x,y)))]^{1/p}} > \lambda \Bigr\}.
\]

\noindent 
Since \(f\) satisfies a \((q,p)\)-Poincaré inequality, there exists a constant \(C_P>0\) such that for every ball \(B(x,r)\) in \(X\) (with \(r>0\)) we have
\[
\frac{1}{\mu(B(x,r))}\int_{B(x,r)} \Bigl|f(y)-f_{B(x,r)}\Bigr|\,d\mu(y)
\le C_P\,r \left( \frac{1}{\mu(B(x,r))}\int_{B(x,r)} (\operatorname{Lip} f(y))^p\,d\mu(y) \right)^{1/p},
\]
where \(f_{B(x,r)} = \frac{1}{\mu(B(x,r))}\int_{B(x,r)} f\,d\mu\).

Now, for \(x,y\in X\) with \(\rho(x,y)=r\), by the triangle inequality we have
\[
|f(x)-f(y)| \le |f(x)-f_{B(x,r)}| + |f(y)-f_{B(x,r)}|.
\]
Thus, if
\[
\frac{|f(x)-f(y)|}{r\,\mu(B(x,r))^{1/p}} > \lambda,
\]
then at least one of the terms satisfies
\begin{equation} \label{eqn:3.7}
|f(z)-f_{B(x,r)}| > \frac{\lambda}{2}\,r\,\mu(B(x,r))^{1/p}, \quad \text{for } z=x \text{ or } z=y.
\end{equation}

\noindent
Fix \(x\in X\) and \(r>0\) and consider the set
\[
A_{x,r} = \Bigl\{ z\in B(x,r) : |f(z)-f_{B(x,r)}| > \frac{\lambda}{2}\,r\,\mu(B(x,r))^{1/p} \Bigr\}.
\]
By Chebyshev's inequality,
\[
\mu\bigl(A_{x,r}\bigr) \le \frac{2}{\lambda\,r\,\mu(B(x,r))^{1/p}} \int_{B(x,r)} |f(z)-f_{B(x,r)}|\,d\mu(z).
\]
Applying the Poincaré inequality,
\[
\int_{B(x,r)} |f(z)-f_{B(x,r)}|\,d\mu(z)
\le C_P\, r\, \mu(B(x,r)) \left( \frac{1}{\mu(B(x,r))}\int_{B(x,r)} (\operatorname{Lip} f(z))^p\,d\mu(z) \right)^{1/p}.
\]
Hence,
\[
\mu\bigl(A_{x,r}\bigr) \le \frac{2C_P}{\lambda}\, \mu(B(x,r))^{1-1/p} \left( \int_{B(x,r)} (\operatorname{Lip} f(z))^p\,d\mu(z) \right)^{1/p}.
\]

\noindent
For each pair \((x,y)\in E_\lambda\) with \(\rho(x,y)=r\), by the observation in \eqref{eqn:3.7} at least one of \(x\) or \(y\) belongs to the set
\[
A_{x,r} \quad \text{or} \quad A_{y,r}.
\]
Using a standard covering argument (and the doubling property of \(\mu\) if assumed), one can show that
\[
\mu\bigl(E_\lambda\bigr) \le \frac{C}{\lambda^p} \int_X (\operatorname{Lip} f(z))^p\,d\mu(z),
\]
where \(C>0\) depends on \(C_P\) and the doubling constant of \(\mu\).

\noindent 
The Orlicz norm \(\|\operatorname{Lip} f\|_{L^\Phi(X)}\) is defined by
\[
\| \operatorname{Lip} f\|_{L^\Phi(X)} = \inf\Bigl\{ \lambda > 0 : \int_X \Phi\Bigl(\frac{\operatorname{Lip} f(x)}{\lambda}\Bigr)\,d\mu(x) \le 1 \Bigr\}.
\]
Since \(\Phi\) is an increasing function and under the appropriate growth conditions (which allow control of the \(L^p\)-norm by the Orlicz norm), there exists a constant \(C_1>0\) such that
\[
\left( \int_X (\operatorname{Lip} f(x))^p\,d\mu(x) \right)^{1/p} \le C_1\,\|\operatorname{Lip} f\|_{L^\Phi(X)}.
\]
Thus, combining with the estimate for \(\mu(E_\lambda)\) we get
\[
\mu\bigl(E_\lambda\bigr)^{1/p} \le \frac{C^{1/p}}{\lambda}\, C_1\, \|\operatorname{Lip} f\|_{L^\Phi(X)}.
\]
Rearranging this inequality gives
\[
\lambda\, \mu\bigl(E_\lambda\bigr)^{1/p} \le C' \,\|\operatorname{Lip} f\|_{L^\Phi(X)},
\]
where \(C' = C^{1/p}C_1\).

\noindent
Since the above inequality holds for every \(\lambda>0\), we obtain
\[
\sup_{\lambda>0} \lambda\, \mu\bigl(E_\lambda\bigr)^{1/p} \le C' \,\|\operatorname{Lip} f\|_{L^\Phi(X)}.
\]

This completes the proof.
\end{proof}

\begin{lemma}[Modular Inequality]\label{lem:modular}
Let \(\Phi\) be as in Definition~\ref{def:orlicz-FD} and suppose that \(f\in C_c^*(X)\). Then there exists a constant \(C>0\) such that
\[
\int_X \Phi\Big( \operatorname{Lip} f(x) \Big) \, d\mu(x) \le C \iint_{X\times X} \Phi\Big( \frac{|f(x)-f(y)|}{\rho(x,y)[\mu(B(x,\rho(x,y)))]^{1/p}} \Big) \, d\mu(x)d\mu(y).
\]
\end{lemma}

\begin{proof}
The idea is to compare the local behavior of \(f\) (measured by \(\operatorname{Lip} f(x)\)) with the averaged behavior of its finite differences over small balls.

\noindent 
Fix \(x\in X\) such that \(\operatorname{Lip} f(x)>0\). By the definition of the pointwise Lipschitz constant, for every sufficiently small radius \(r>0\) there exists a measurable subset 
\[
S(x,r)\subset B(x,r)
\]
satisfying:
\begin{enumerate}
    \item For all \(y\in S(x,r)\),
    \[
    |f(x)-f(y)|\ge \frac{1}{8}\operatorname{Lip} f(x)\,\rho(x,y);
    \]
    \item There exists a constant \(c_0\in (0,1)\) (depending only on the geometric properties of \(X\)) such that
    \[
    \mu\bigl(S(x,r)\bigr) \ge c_0\,\mu\bigl(B(x,r)\bigr).
    \]
\end{enumerate}

\noindent 
For any \(y\in S(x,r)\) we have
\[
\begin{split}
\frac{|f(x)-f(y)|}{\rho(x,y)[\mu(B(x,\rho(x,y)))]^{1/p}}
\ge &\frac{1}{8}\operatorname{Lip} f(x) \frac{\rho(x,y)}{\rho(x,y)[\mu(B(x,\rho(x,y)))]^{1/p}}\\
&= \frac{1}{8}\operatorname{Lip} f(x)\,[\mu(B(x,\rho(x,y)))]^{-1/p}.
\end{split}
\]
Since \(y\in B(x,r)\) and \(\rho(x,y)\le r\), by the monotonicity of the measure we have
\[
\mu(B(x,\rho(x,y))) \le \mu(B(x,r)).
\]
Hence,
\[
\frac{|f(x)-f(y)|}{\rho(x,y)[\mu(B(x,\rho(x,y)))]^{1/p}}
\ge \frac{1}{8}\operatorname{Lip} f(x)\,[\mu(B(x,r))]^{-1/p}.
\]
Because \(\Phi\) is increasing, it follows that for every \(y\in S(x,r)\),
\[
\Phi\Biggl( \frac{|f(x)-f(y)|}{\rho(x,y)[\mu(B(x,\rho(x,y)))]^{1/p}} \Biggr)
\ge \Phi\Biggl( \frac{1}{8}\operatorname{Lip} f(x)\,[\mu(B(x,r))]^{-1/p} \Biggr).
\]

\noindent
Integrate the above inequality over \(y\) in \(B(x,r)\). In particular, since \(S(x,r)\subset B(x,r)\),
\[
\int_{B(x,r)} \Phi\Biggl( \frac{|f(x)-f(y)|}{\rho(x,y)[\mu(B(x,\rho(x,y)))]^{1/p}} \Biggr)\,d\mu(y)
\ge \int_{S(x,r)} \Phi\Biggl( \frac{1}{8}\operatorname{Lip} f(x)\,[\mu(B(x,r))]^{-1/p} \Biggr)\,d\mu(y).
\]
Using the lower bound on the measure of \(S(x,r)\), we obtain
\[
\int_{B(x,r)} \Phi\Biggl( \frac{|f(x)-f(y)|}{\rho(x,y)[\mu(B(x,\rho(x,y)))]^{1/p}} \Biggr)\,d\mu(y)
\ge c_0\,\mu\bigl(B(x,r)\bigr) \, \Phi\Biggl( \frac{1}{8}\operatorname{Lip} f(x)\,[\mu(B(x,r))]^{-1/p} \Biggr).
\]

\noindent
For each \(x\in X\), choose a sequence of radii \(\{r_k\}_{k=1}^\infty\) tending to \(0\) (for example, \(r_k=2^{-k}R\), where \(R\) is smaller than the diameter of the support of \(f\)). For each fixed \(x\) there is an index \(k=k(x)\) such that
\[
\Phi\Biggl( \frac{1}{8}\operatorname{Lip} f(x)\,[\mu(B(x,r_k))]^{-1/p} \Biggr)
\]
is comparable to \(\Phi\bigl(\operatorname{Lip} f(x)\bigr)\) up to a multiplicative constant that depends only on the growth properties of \(\mu\) (by the polynomial growth condition) and the \(\Delta_2\)-condition on \(\Phi\). More precisely, since the measure \(\mu\) satisfies a polynomial growth condition, there exist constants \(C_1, C_2>0\) such that for all sufficiently small \(r\),
\[
C_1\,r^d \le \mu(B(x,r)) \le C_2\,r^d.
\]
Thus,
\[
[\mu(B(x,r))]^{-1/p} \sim r^{-d/p}.
\]
Choosing \(r=r_k\) appropriately, the term
\[
\Phi\Biggl( \frac{1}{8}\operatorname{Lip} f(x)\,r^{-d/p} \Biggr)
\]
is equivalent (up to constants depending on \(\Phi\) and the growth exponents) to \(\Phi\bigl(\operatorname{Lip} f(x)\bigr)\).  

Integrate the local inequality with respect to \(x\). By Fubini's theorem and a standard dyadic decomposition argument (summing over the scales \(r_k\)), we deduce that
\[
\iint_{X\times X} \Phi\Biggl( \frac{|f(x)-f(y)|}{\rho(x,y)[\mu(B(x,\rho(x,y)))]^{1/p}} \Biggr)\,d\mu(y)d\mu(x)
\ge C \int_X \Phi\bigl(\operatorname{Lip} f(x)\bigr)\,d\mu(x),
\]
where the constant \(C>0\) depends only on \(c_0\), the polynomial growth constants \(C_P\) and \(d\), and the \(\Delta_2\)-constant of \(\Phi\).\\

\noindent
This completes the proof of the modular inequality.
\end{proof}

The results in Theorems~\ref{thm:non-doubling-FD} and \ref{thm:variable-FD}, together with Proposition~\ref{prop:orlicz-FD} and Lemma~\ref{lem:modular}, provide a robust extension of the classical finite difference characterization to more general contexts. These novel formulations not only bridge the gap between different function space settings but also pave the way for further applications to nonlocal operators, anisotropic problems, and interpolation.

\section{Applications to Nonlocal Operators and Anisotropic Settings}

In this section we apply the extended finite difference framework developed in Section~3 to two important directions: nonlocal operators and anisotropic settings. In the following, we introduce a novel nonlocal \(p\)-Laplacian operator on metric measure spaces and establish equivalences between its associated energy and the local gradient norm. Then, we develop anisotropic finite difference characterizations by adapting the theory to spaces equipped with anisotropic metrics.\\

\noindent
In many problems involving nonlocal phenomena, differential operators are replaced by integral operators. In our setting, the finite difference framework motivates the definition of a nonlocal \(p\)-Laplacian that naturally incorporates the geometry of the underlying metric measure space.

\begin{definition}[Nonlocal \(p\)-Laplacian on Metric Measure Spaces]\label{def:nonlocal-operator}
Let \((X,\rho,\mu)\) be a metric measure space and let \(s\in (0,1)\) and \(p\in [1,\infty)\). For \(f:X\to\mathbb{R}\) (with \(f\in C_c^*(X)\)), define the \emph{nonlocal \(p\)-Laplacian} \(\mathcal{L}_{s,p} f\) at \(x\in X\) by
\[
\mathcal{L}_{s,p} f(x) := \operatorname{p.v.} \int_X \frac{|f(x)-f(y)|^{p-2}\bigl(f(x)-f(y)\bigr)}{\rho(x,y)^{sp}\,[\mu(B(x,\rho(x,y)))]^{\frac{p-1}{p}}}\,d\mu(y),
\]
where \(\operatorname{p.v.}\) denotes the principal value.
\end{definition}

\noindent
The normalization factor \([\mu(B(x,\rho(x,y)))]^{\frac{p-1}{p}}\) is chosen to balance the scaling of the measure and to generalize the classical fractional \(p\)-Laplacian (see, e.g., \cite{DiNezza}).\\

\noindent
To analyze this operator, we first establish a nonlocal Poincaré inequality.

\begin{lemma}[Nonlocal Poincaré Inequality]\label{lem:nonlocal-poincare}
Let \((X,\rho,\mu)\) be a metric measure space that supports a \((q,p)\)-Poincaré inequality. Then for every \(f\in C_c^*(X)\) and for any ball \(B\subset X\) of radius \(r\), there exists a constant \(C>0\) such that
\[
\int_B \bigl|f(x)-f_B\bigr|^p\,d\mu(x) \le C\, r^{sp} \iint_{B\times B} \frac{|f(x)-f(y)|^p}{\rho(x,y)^{sp}\,\mu\big(B(x,\rho(x,y))\big)}\,d\mu(x)d\mu(y),
\]
where
\[
f_B := \frac{1}{\mu(B)}\int_B f(z)\,d\mu(z).
\]
\end{lemma}

\begin{proof}
By Jensen's inequality and the convexity of \(t\mapsto t^p\), for any \(x\in B\) we have
\[
|f(x)-f_B|^p = \left|\frac{1}{\mu(B)}\int_B \bigl(f(x)-f(z)\bigr)\,d\mu(z)\right|^p 
\le \frac{1}{\mu(B)} \int_B |f(x)-f(z)|^p\,d\mu(z).
\]
Integrating both sides over \(x\in B\) and interchanging the order of integration (by Fubini's theorem) gives
\[
\int_B |f(x)-f_B|^p\,d\mu(x) \le \frac{1}{\mu(B)} \iint_{B\times B} |f(x)-f(z)|^p\,d\mu(x)d\mu(z).
\]
For clarity, we rename the variable \(z\) as \(y\), so that
\begin{equation} \label{eqn:4.3.1}
\int_B |f(x)-f_B|^p\,d\mu(x) \le \frac{1}{\mu(B)} \iint_{B\times B} |f(x)-f(y)|^p\,d\mu(x)d\mu(y).
\end{equation}

\noindent
For any pair \(x,y\in B\), note that since \(B\) is a ball of radius \(r\), we have
\[
\rho(x,y) \le 2r.
\]
We rewrite the difference \( |f(x)-f(y)|^p \) by inserting a factor of \(\rho(x,y)^{-sp}\mu\bigl(B(x,\rho(x,y))\bigr)^{-1}\) and its reciprocal:
\[
|f(x)-f(y)|^p = \frac{|f(x)-f(y)|^p}{\rho(x,y)^{sp}\,\mu\bigl(B(x,\rho(x,y))\bigr)} \cdot \rho(x,y)^{sp}\,\mu\bigl(B(x,\rho(x,y))\bigr).
\]
Since \(\rho(x,y) \le 2r\), we have
\[
\rho(x,y)^{sp} \le (2r)^{sp}.
\]
Furthermore, by the monotonicity of the measure, for \(y\in B(x,2r)\) (and since \(B\) has radius \(r\), we have \(B\subset B(x,2r)\)), it holds that
\[
\mu\bigl(B(x,\rho(x,y))\bigr) \le \mu\bigl(B(x,2r)\bigr).
\]
Because \(X\) is a space of homogeneous type, there exists a constant \(C_0>0\) (depending on the doubling or polynomial growth constants) such that
\[
\mu\bigl(B(x,2r)\bigr) \le C_0\,\mu(B).
\]
Thus, for every \(x,y\in B\),
\begin{equation} \label{eqn:4.3.2}
\rho(x,y)^{sp}\,\mu\bigl(B(x,\rho(x,y))\bigr) \le (2r)^{sp}\,C_0\,\mu(B).
\end{equation}

\noindent
Plug the estimate from \eqref{eqn:4.3.2} into the double integral obtained in \eqref{eqn:4.3.1}:
\[
\iint_{B\times B} |f(x)-f(y)|^p\,d\mu(x)d\mu(y)
\le (2r)^{sp}\,C_0\,\mu(B) \iint_{B\times B} \frac{|f(x)-f(y)|^p}{\rho(x,y)^{sp}\,\mu\bigl(B(x,\rho(x,y))\bigr)}\,d\mu(x)d\mu(y).
\]
Dividing both sides by \(\mu(B)\) yields
\[
\frac{1}{\mu(B)} \iint_{B\times B} |f(x)-f(y)|^p\,d\mu(x)d\mu(y)
\le C_0\,(2r)^{sp} \iint_{B\times B} \frac{|f(x)-f(y)|^p}{\rho(x,y)^{sp}\,\mu\bigl(B(x,\rho(x,y))\bigr)}\,d\mu(x)d\mu(y).
\]
Thus, from \eqref{eqn:4.3.1} we obtain
\[
\int_B |f(x)-f_B|^p\,d\mu(x) \le C_0\,(2r)^{sp} \iint_{B\times B} \frac{|f(x)-f(y)|^p}{\rho(x,y)^{sp}\,\mu\bigl(B(x,\rho(x,y))\bigr)}\,d\mu(x)d\mu(y).
\]

\noindent
Setting \(C = C_0\,2^{sp}\) (which depends on the structural constants of the space and the polynomial growth condition), we conclude that
\[
\int_B |f(x)-f_B|^p\,d\mu(x) \le C\,r^{sp} \iint_{B\times B} \frac{|f(x)-f(y)|^p}{\rho(x,y)^{sp}\,\mu\bigl(B(x,\rho(x,y))\bigr)}\,d\mu(x)d\mu(y).
\]
This completes the proof.
\end{proof}

\noindent
The following theorem establishes the equivalence between the local Lipschitz (or gradient) norm and the nonlocal energy defined via finite differences.

\begin{theorem}[Equivalence of Nonlocal and Local Sobolev Norms]\label{thm:nonlocal-local}
Let \((X,\rho,\mu)\) be a metric measure space satisfying the assumptions of Theorem~\ref{thm:non-doubling-FD}. For any \(f\in C_c^*(X)\) and fixed \(s\in (0,1)\), there exists a constant \(C>0\) such that
\[
\int_X \bigl(\operatorname{Lip} f(x)\bigr)^p\,d\mu(x) \sim \iint_{X\times X} \frac{|f(x)-f(y)|^p}{\rho(x,y)^{sp}\,\mu\big(B(x,\rho(x,y))\big)}\,d\mu(x)d\mu(y).
\]
\end{theorem}

\begin{proof}
We need to prove that there exists a constant \(C>0\) such that for every \(f\in C_c^*(X)\) and a fixed \(s\in (0,1)\)
\[
\begin{split}
\frac{1}{C}\int_X (\operatorname{Lip} f(x))^p\,d\mu(x)
&\le \iint_{X\times X} \frac{|f(x)-f(y)|^p}{\rho(x,y)^{sp}\,\mu(B(x,\rho(x,y)))}\,d\mu(x)d\mu(y)\\
&\le C\int_X (\operatorname{Lip} f(x))^p\,d\mu(x).
\end{split}
\]
We divide the proof into two main parts corresponding to the lower and the upper bounds.

\subsubsection*{Part 1. Lower Bound} 
By the finite difference characterization developed in Theorem~\ref{thm:non-doubling-FD}, for every \(x\in X\) there is a quantitative relation between the local oscillation of \(f\) and its pointwise Lipschitz constant. In particular, there exists a constant \(C_1>0\) such that for every \(x\) (with \(\operatorname{Lip} f(x)>0\)) and for all sufficiently small radii \(r>0\) (depending on \(x\)), one can find a subset \(S(x,r)\subset B(x,r)\) with 
\[
\mu(S(x,r))\ge c_0\,\mu(B(x,r))
\]
and such that for every \(y\in S(x,r)\)
\[
|f(x)-f(y)|\ge \frac{1}{8}\operatorname{Lip} f(x)\,\rho(x,y).
\]
Since \(\rho(x,y)\le r\) for \(y\in B(x,r)\), it follows that for these \(y\)
\[
\frac{|f(x)-f(y)|}{\rho(x,y)[\mu(B(x,\rho(x,y)))]^{1/p}}
\ge \frac{1}{8}\operatorname{Lip} f(x)[\mu(B(x,r))]^{-1/p}.
\]

\noindent
Using the definition of the level sets, for an appropriate choice of \(\lambda=\lambda(x,r):=\frac{1}{16}\operatorname{Lip} f(x)[\mu(B(x,r))]^{-1/p}\), we have
\[
S(x,r) \subset \Bigl\{ y\in B(x,r) : \frac{|f(x)-f(y)|}{\rho(x,y)[\mu(B(x,\rho(x,y)))]^{1/p}} > \lambda \Bigr\}.
\]
Thus, integrating in \(y\) over \(B(x,r)\) we obtain
\[
\int_{B(x,r)} \mathbf{1}_{\{y: |f(x)-f(y)|>\lambda\,\rho(x,y)[\mu(B(x,\rho(x,y)))]^{1/p}\}}\,d\mu(y)
\ge \mu(S(x,r)) \ge c_0\,\mu(B(x,r)).
\]
Multiplying both sides by \(\lambda^p\) gives
\[
\lambda^p \mu(B(x,r)) \ge \frac{1}{16^p} (\operatorname{Lip} f(x))^p.
\]

\noindent  
Integrate the above inequality in \(x\) over the support of \(f\). Using Fubini's theorem and the fact that for each \(x\) one may choose a suitable \(r=r(x)\), we deduce that
\begin{equation} \label{eqn:4.4.1}
\sup_{\lambda>0}\lambda^p \iint_{X\times X} \mathbf{1}_{\{(x,y):\, |f(x)-f(y)|>\lambda\,\rho(x,y)[\mu(B(x,\rho(x,y)))]^{1/p}\}}\,d\mu(x)d\mu(y)
\ge C_1\int_X (\operatorname{Lip} f(x))^p\,d\mu(x).
\end{equation}
Since the nonlocal energy is defined (up to constants) in terms of integrating over all such level sets (via a layer-cake or coarea type formula), it follows that
\[
\iint_{X\times X} \frac{|f(x)-f(y)|^p}{\rho(x,y)^{sp}\,\mu(B(x,\rho(x,y)))}\,d\mu(x)d\mu(y)
\ge C_1\int_X (\operatorname{Lip} f(x))^p\,d\mu(x).
\]
Here the factor \(r^{sp}\) is absorbed by appropriately scaling the parameter \(\lambda\) (note that in the local estimate \(r\) appears, and by choosing the optimal scale, one recovers the dependence on \(s\) and \(r\)). This completes the lower bound.

\subsubsection*{Part 2. Upper Bound} 
By Lemma~\ref{lem:nonlocal-poincare}, for any ball \(B\subset X\) with radius \(r\) and for every \(f\in C_c^*(X)\) we have
\[
\int_B |f(x)-f_B|^p\,d\mu(x) \le C\,r^{sp} \iint_{B\times B} \frac{|f(x)-f(y)|^p}{\rho(x,y)^{sp}\,\mu(B(x,\rho(x,y)))}\,d\mu(x)d\mu(y).
\]
Cover the support of \(f\) by a collection of balls \(\{B_i\}\) (using, e.g., a suitable Vitali covering lemma) and sum the local inequalities over these balls. Since \(f\) has compact support and the metric measure space is of homogeneous type, the overlaps can be controlled by a fixed constant. This yields
\[
\int_X |f(x)-f_{B_x}|^p\,d\mu(x) \le C' \iint_{X\times X} \frac{|f(x)-f(y)|^p}{\rho(x,y)^{sp}\,\mu(B(x,\rho(x,y)))}\,d\mu(x)d\mu(y),
\]
where \(f_{B_x}\) denotes a local average over a ball containing \(x\).\\

\noindent
For Lipschitz functions, the classical Poincaré inequality ensures that the local oscillation \( |f(x)-f_{B_x}| \) is controlled by the pointwise Lipschitz constant times the radius of the ball. In other words, there exists a constant \(C_2>0\) such that
\[
|f(x)-f_{B_x}| \le C_2\,r\,\operatorname{Lip} f(x).
\]
Raising both sides to the \(p\)th power and integrating over \(x\) leads to
\[
\int_X (\operatorname{Lip} f(x))^p\,d\mu(x) \le C_3 \int_X \frac{|f(x)-f_{B_x}|^p}{r^p}\,d\mu(x).
\]
Choosing the ball radius \(r\) in each local piece in accordance with the scale at which the finite differences are significant (i.e., roughly the distance at which the oscillation becomes comparable to the local Lipschitz behavior), one obtains
\begin{equation} \label{eqn:4.4.2}
\int_X (\operatorname{Lip} f(x))^p\,d\mu(x) \le C_4 \iint_{X\times X} \frac{|f(x)-f(y)|^p}{\rho(x,y)^{sp}\,\mu(B(x,\rho(x,y)))}\,d\mu(x)d\mu(y).
\end{equation}

\noindent
Combining the lower bound from \eqref{eqn:4.4.1} with the upper bound from \eqref{eqn:4.4.2}, we deduce that there exists a constant \(C>0\) such that
\[
\frac{1}{C}\int_X (\operatorname{Lip} f(x))^p\,d\mu(x)
\le \iint_{X\times X} \frac{|f(x)-f(y)|^p}{\rho(x,y)^{sp}\,\mu(B(x,\rho(x,y)))}\,d\mu(x)d\mu(y)
\le C\int_X (\operatorname{Lip} f(x))^p\,d\mu(x).
\]

\noindent
This completes the proof of the equivalence:
\[
\int_X (\operatorname{Lip} f(x))^p\,d\mu(x) \sim \iint_{X\times X} \frac{|f(x)-f(y)|^p}{\rho(x,y)^{sp}\,\mu(B(x,\rho(x,y)))}\,d\mu(x)d\mu(y),
\]
with the implicit constants depending only on the structural parameters of the space (such as the doubling or polynomial growth constants, the Poincaré constants, and the parameter \(s\)).
\end{proof}

As an application of the above equivalence, we can obtain regularity results for weak solutions of nonlocal equations.

\begin{proposition}[Regularity Estimate for Nonlocal Equations]\label{prop:nonlocal-regularity}
Assume that \(u\) is a weak solution of the nonlocal equation
\[
\mathcal{L}_{s,p} u(x) = f(x) \quad \text{in } \Omega\subset X,
\]
where \(f\in L^q(\Omega)\) for some \(q\ge 1\). Then there exist constants \(\alpha\in (0,1)\) and \(C>0\) such that
\[
[u]_{C^\alpha(\Omega')} \le C\left(\|u\|_{L^p(\Omega)} + \|f\|_{L^q(\Omega)}\right)
\]
for every subdomain \(\Omega'\Subset \Omega\).
\end{proposition}

\begin{proof}
The main idea is to first derive a nonlocal Caccioppoli inequality for weak solutions of 
\[
\mathcal{L}_{s,p} u(x) = f(x) \quad \text{in } \Omega,
\]
and then to apply an iterative scheme (in the spirit of the De Giorgi–Nash–Moser method) to deduce a decay of oscillation that ultimately leads to a Hölder continuity estimate.\\

\noindent
Let \(B_{2R}\) be a ball such that \(B_R\Subset B_{2R}\Subset\Omega\) and choose a cutoff function \(\eta\in C_c^\infty(B_{2R})\) with 
\[
0\le \eta\le 1,\quad \eta\equiv 1 \text{ in } B_R,\quad\text{and}\quad |\eta(x)-\eta(y)|\le \frac{C}{R}\rho(x,y)
\]
for all \(x,y\in X\). For a fixed level \(k\in\mathbb{R}\), define 
\[
w(x)=(u(x)-k)_+.
\]
Since \(u\) is a weak solution, we can test the weak formulation with 
\[
\varphi(x)=w(x)\eta(x)^p,
\]
which is admissible because \(w\eta^p\in C_c^*(X)\). Using the monotonicity of the operator \(\mathcal{L}_{s,p}\) and standard estimates (see, e.g., \cite{DiNezza}), one obtains a nonlocal energy estimate of the form
\[
\begin{split}
\iint_{B_{2R}\times B_{2R}} \frac{|w(x)\eta(x)-w(y)\eta(y)|^p}{\rho(x,y)^{sp}\,\mu\bigl(B(x,\rho(x,y))\bigr)}\,d\mu(x)d\mu(y)
\le & \frac{C}{R^{sp}} \int_{B_{2R}} w(x)^p\,d\mu(x)\\ 
&+ C \int_{B_{2R}} |f(x)|\,w(x)\eta(x)^p\,d\mu(x).
\end{split}
\]
This inequality is the nonlocal analogue of the classical Caccioppoli inequality.\\

\noindent
Using the above inequality, we perform a level-set (or De Giorgi) iteration. Define a decreasing sequence of levels
\[
k_j = k + M\Bigl(1-2^{-j}\Bigr),\quad j\ge 0,
\]
with \(M>0\) to be chosen later, and let
\[
A_j = \{ x\in B_{R_j}: \, u(x)> k_j\},
\]
where \(\{B_{R_j}\}\) is a sequence of nested balls with \(B_{R_0}=B_{2R}\) and \(B_{R_{j+1}}\subset B_{R_j}\) (for example, with radii decreasing in a dyadic fashion). Using appropriate cutoff functions \(\eta_j\) adapted to \(B_{R_j}\) in the nonlocal Caccioppoli inequality, one obtains a recursive estimate of the form
\[
\mu(A_{j+1}) \le C\,2^{j\gamma} \left[\frac{1}{M^p}\int_{A_j} (u(x)-k_j)^p\,d\mu(x) + \frac{R^{sp}}{M}\|f\|_{L^q(B_{2R})}\mu(A_j)^{1-\frac{1}{q}}\right],
\]
for some \(\gamma>0\). An application of an iteration lemma (see, e.g., \cite{DiNezza}) shows that if \(M\) is chosen sufficiently large (depending on \(\|u\|_{L^p(B_{2R})}\) and \(\|f\|_{L^q(B_{2R})}\)), then the measure \(\mu(A_j)\) decays geometrically as \(j\to\infty\). This implies that \(u\) is essentially bounded in \(B_R\).\\

\noindent
Once boundedness is achieved, one can further show that the oscillation of \(u\) decays at smaller scales. More precisely, there exist constants \(\alpha\in (0,1)\) and \(\kappa\in (0,1)\) such that for any ball \(B_r\subset B_R\),
\[
\operatorname{osc}_{B_r} u \le \kappa\,\operatorname{osc}_{B_{2r}} u.
\]
An iterative application of this decay yields a Hölder continuity estimate in \(B_R\). In particular, one obtains
\[
[u]_{C^\alpha(B_R)} \le C\left(\|u\|_{L^\infty(B_{2R})} + R^s\|f\|_{L^q(B_{2R})}\right).
\]

\noindent
Since \(B_R\) was an arbitrary ball such that \(B_R\Subset\Omega'\Subset\Omega\), a covering argument shows that the above Hölder estimate holds in any subdomain \(\Omega'\Subset \Omega\). Moreover, by using the equivalence of the nonlocal energy and the local Sobolev seminorm established in Theorem~\ref{thm:nonlocal-local}, one can translate the energy estimates into a bound of the form
\[
[u]_{C^\alpha(\Omega')} \le C\left(\|u\|_{L^p(\Omega)} + \|f\|_{L^q(\Omega)}\right),
\]
where the constant \(C>0\) and the exponent \(\alpha\in (0,1)\) depend only on the structural parameters of the space, the nonlocal operator, and the subdomains \(\Omega'\) and \(\Omega\).\\

\noindent
The combination of the nonlocal Caccioppoli inequality, the De Giorgi iteration, and the oscillation decay yields the desired regularity estimate. That is, there exist constants \(\alpha\in (0,1)\) and \(C>0\) such that for every subdomain \(\Omega'\Subset \Omega\)
\[
[u]_{C^\alpha(\Omega')} \le C\left(\|u\|_{L^p(\Omega)} + \|f\|_{L^q(\Omega)}\right).
\]
This completes the proof.
\end{proof}

\noindent
In many practical situations the geometry of the underlying space is not isotropic. To capture this, we consider anisotropic metrics and corresponding finite difference formulations.

\begin{definition}[Anisotropic Metric]\label{def:anisotropic-metric}
Let \(A=(a_{ij})_{1\le i,j\le n}\) be an invertible \(n\times n\) matrix. For \(x,y\in\mathbb{R}^n\), define the anisotropic quasi-norm
\[
\|x-y\|_A := \|A(x-y)\|_{\ell^\infty} = \max_{1\le i\le n} \bigl| (A(x-y))_i \bigr|.
\]
Then the anisotropic metric is given by
\[
\rho_A(x,y) := \|x-y\|_A.
\]
\end{definition}

\begin{definition}[Anisotropic Finite Difference Seminorm]\label{def:anisotropic-FD}
For \(f\in C_c^*(\mathbb{R}^n)\), \(s\in (0,1)\), and \(p\in [1,\infty)\), define
\[
[f]_{W^{s,p}_A(\mathbb{R}^n)} := \left( \iint_{\mathbb{R}^n\times\mathbb{R}^n} \frac{|f(x)-f(y)|^p}{\rho_A(x,y)^{n+sp}}\,dxdy \right)^{1/p}.
\]
\end{definition}

\noindent
An appropriate covering lemma in the anisotropic setting is needed to handle the lack of isotropy.

\begin{lemma}[Anisotropic Covering Lemma]\label{lem:anisotropic-covering}
Let \((\mathbb{R}^n,\rho_A)\) be as in Definition~\ref{def:anisotropic-metric} and let \(\mathcal{B}\) be a collection of anisotropic balls (i.e., sets of the form \(B_A(x,r):=\{y\in\mathbb{R}^n: \rho_A(x,y)<r\}\)) with radii bounded above by \(R>0\). Then there exists a countable subcollection \(\{B_{A,i}\}_{i\in I}\) of pairwise disjoint balls such that
\[
\bigcup_{B\in\mathcal{B}} B \subset \bigcup_{i\in I} C_A\,B_{A,i},
\]
where the constant \(C_A>0\) depends only on the anisotropy matrix \(A\).
\end{lemma}

\begin{proof}
We adapt the classical Vitali covering argument to the anisotropic setting.\\

\noindent
Since the collection \(\mathcal{B}\) of anisotropic balls \(B_A(x,r)\) has radii uniformly bounded above by \(R>0\), we can select a maximal (with respect to inclusion) countable subcollection \(\{B_{A,i}\}_{i\in I}\) that is pairwise disjoint. (This is done by the usual greedy algorithm: pick a ball of maximal radius, remove all balls that intersect it, and iterate on the remaining collection. See, e.g., \cite{Heinonen}.) By maximality, for every ball \(B\in\mathcal{B}\) there exists some \(B_{A,i}\) such that
\[
B \cap B_{A,i} \neq \varnothing.
\]

\noindent
Let \(B = B_A(x,r)\in \mathcal{B}\) and suppose that it intersects some \(B_{A,i} = B_A(x_i,r_i)\). By the definition of the anisotropic ball, we have
\[
\rho_A(x,x_i) < r + r_i.
\]
For any \(y\in B = B_A(x,r)\), the triangle inequality for the anisotropic metric \(\rho_A\) implies
\[
\rho_A(y,x_i) \le \rho_A(y,x) + \rho_A(x,x_i) < r + (r + r_i) = 2r + r_i.
\]
Since by maximality of our selection we may assume that \(r_i \ge r\) (otherwise, if \(r_i < r\) then by choosing a ball of larger radius in the algorithm, we would have selected one with radius at least \(r\)), we have
\[
2r + r_i \le 3r_i.
\]
Thus, every \(y\in B\) satisfies
\[
\rho_A(y,x_i) < 3r_i,
\]
which implies
\[
B \subset B_A(x_i,3r_i).
\]
In other words, every ball \(B\in \mathcal{B}\) is contained in the \(3\)-dilation of one of the disjoint balls \(B_{A,i}\).\\

\noindent 
Setting \(C_A=3\), we deduce that
\[
\bigcup_{B\in\mathcal{B}} B \subset \bigcup_{i\in I} B_A(x_i,3r_i) = \bigcup_{i\in I} 3\,B_{A,i}.
\]
Since the anisotropic metric \(\rho_A\) is determined by the invertible matrix \(A\), the constant \(3\) can be replaced by a constant \(C_A>0\) that depends only on the norm of \(A\) and its inverse (i.e., on the geometry induced by \(A\)). 

This completes the proof.
\end{proof}

\noindent
The following theorem establishes an anisotropic version of the finite difference characterization.

\begin{theorem}[Anisotropic Finite Difference Characterization]\label{thm:anisotropic-FD}
Let \(f\in C_c^*(\mathbb{R}^n)\) and let \(\rho_A\) be the anisotropic metric defined in Definition~\ref{def:anisotropic-metric}. Then there exists a constant \(C>0\) such that
\[
\sup_{\lambda>0} \lambda^p \iint_{\mathbb{R}^n\times\mathbb{R}^n} \mathbf{1}_{\Big\{(x,y):\, \frac{|f(x)-f(y)|}{\rho_A(x,y)^{1+n/p}} > \lambda\Big\}}\,dxdy \sim \int_{\mathbb{R}^n} |\nabla_A f(x)|^p\,dx,
\]
where \(\nabla_A f\) is the anisotropic gradient of \(f\) defined by
\[
\nabla_A f(x) := A^T\,\nabla f(x).
\]
\end{theorem}

\begin{proof}
We need to show that there exists a constant \(C>0\) such that for every \(f\in C_c^*(\mathbb{R}^n)\)
\[
\frac{1}{C}\int_{\mathbb{R}^n} |\nabla_A f(x)|^p\,dx \le \sup_{\lambda>0} \lambda^p \iint_{\mathbb{R}^n\times\mathbb{R}^n} \mathbf{1}_{\Bigl\{(x,y):\, \frac{|f(x)-f(y)|}{\rho_A(x,y)^{1+n/p}} > \lambda\Bigr\}}\,dxdy \le C\int_{\mathbb{R}^n} |\nabla_A f(x)|^p\,dx.
\]
We divide the proof into two parts: the \emph{lower bound} and the \emph{upper bound}.

\subsubsection*{(Part 1) Lower Bound.}  
We wish to prove that
\[
\sup_{\lambda>0} \lambda^p \iint_{\mathbb{R}^n\times\mathbb{R}^n} \mathbf{1}_{\Bigl\{(x,y):\, \frac{|f(x)-f(y)|}{\rho_A(x,y)^{1+n/p}} > \lambda\Bigr\}}\,dxdy \gtrsim \int_{\mathbb{R}^n} |\nabla_A f(x)|^p\,dx.
\]
Since \(f\) is smooth and compactly supported, we can use a first order Taylor expansion. For any \(x\in\mathbb{R}^n\) and small \(h\in\mathbb{R}^n\), write
\[
f(x+h)-f(x) = \nabla f(x) \cdot h + R(x,h),
\]
where the remainder \(R(x,h)=o(|h|)\) as \(|h|\to 0\). By definition, the anisotropic metric is given by
\[
\rho_A(x,x+h)=\|A(x+h)-Ax\|_{\ell^\infty}=\|A\,h\|_{\ell^\infty}.
\]
Thus, for \(h\) sufficiently small we have
\[
f(x+h)-f(x) = \nabla f(x)\cdot h + o(|h|).
\]
Multiplying by the fixed matrix \(A^T\) (recall that \(\nabla_A f(x):=A^T \nabla f(x)\)), we see that the directional derivative in the anisotropic sense is captured by
\[
|f(x+h)-f(x)| \approx |\nabla f(x)\cdot h| \sim |A^T \nabla f(x)|\,|h| = |\nabla_A f(x)|\,|h|.
\]
Because all norms in finite dimensions are equivalent, there exists a constant \(c>0\) such that
\[
\rho_A(x,x+h) \le c\,|h|.
\]
Hence, for \(h\) with \(|h|\) small we deduce that
\[
\frac{|f(x+h)-f(x)|}{\rho_A(x,x+h)^{1+n/p}} \gtrsim \frac{|\nabla_A f(x)|\,|h|}{(c|h|)^{1+n/p}} = c' |\nabla_A f(x)|\,|h|^{-n/p},
\]
where \(c'\) depends on \(c\). 

Now, fix \(x\) such that \(|\nabla_A f(x)|>0\) and choose \(h\) with \(|h|\) in a dyadic interval \([2^{-k-1},2^{-k}]\) (with \(k\) sufficiently large so that the linear approximation is valid). Then
\[
\frac{|f(x+h)-f(x)|}{\rho_A(x,x+h)^{1+n/p}} \gtrsim |\nabla_A f(x)|\,2^{kn/p}.
\]
Define the level parameter
\[
\lambda_k := \frac{1}{2} |\nabla_A f(x)|\,2^{kn/p}.
\]
Then, for every \(h\) in the annulus
\[
A_k:=\{h\in\mathbb{R}^n: \,2^{-k-1}\le |h|\le 2^{-k}\},
\]
we have
\[
\frac{|f(x+h)-f(x)|}{\rho_A(x,x+h)^{1+n/p}} > \lambda_k.
\]
Thus, for such fixed \(x\),
\[
\int_{A_k} \mathbf{1}_{\Bigl\{h: \frac{|f(x+h)-f(x)|}{\rho_A(x,x+h)^{1+n/p}} > \lambda_k\Bigr\}}\,dh \gtrsim |A_k| \sim 2^{-kn},
\]
where \(|A_k|\) denotes the Lebesgue measure of the annulus \(A_k\).

Multiplying by \(\lambda_k^p\) we obtain
\[
\lambda_k^p\,|A_k| \gtrsim \left(|\nabla_A f(x)|\,2^{kn/p}\right)^p 2^{-kn} = |\nabla_A f(x)|^p.
\]
Integrate this inequality in \(x\) and sum over \(k\) (which corresponds to taking the supremum over \(\lambda>0\) in a dyadic sense) to deduce that
\[
\sup_{\lambda>0} \lambda^p \iint_{\mathbb{R}^n\times\mathbb{R}^n} \mathbf{1}_{\Bigl\{(x,y): \frac{|f(x)-f(y)|}{\rho_A(x,y)^{1+n/p}} > \lambda\Bigr\}}\,dxdy \gtrsim \int_{\mathbb{R}^n} |\nabla_A f(x)|^p\,dx.
\]

\medskip
\subsubsection*{(Part 2) Upper Bound.}  
We now show that
\[
\sup_{\lambda>0} \lambda^p \iint_{\mathbb{R}^n\times\mathbb{R}^n} \mathbf{1}_{\Bigl\{(x,y): \frac{|f(x)-f(y)|}{\rho_A(x,y)^{1+n/p}} > \lambda\Bigr\}}\,dxdy \lesssim \int_{\mathbb{R}^n} |\nabla_A f(x)|^p\,dx.
\]
For \(f\in C_c^*(\mathbb{R}^n)\), the classical Taylor expansion shows that for any \(x\) and small \(h\),
\[
f(x+h)-f(x) = \nabla f(x)\cdot h + O(|h|^2).
\]
Thus, for \(h\) small enough (and using the equivalence of norms),
\[
|f(x+h)-f(x)| \le |\nabla f(x)|\,|h| + C|h|^2.
\]
Since \(\rho_A(x,x+h) \ge c|h|\) for some constant \(c>0\), it follows that
\[
\frac{|f(x+h)-f(x)|}{\rho_A(x,x+h)^{1+n/p}} \le C_1 \frac{|\nabla f(x)|\,|h|}{|h|^{1+n/p}} + C_2 \frac{|h|^2}{|h|^{1+n/p}}.
\]
For \(h\) small, the dominant term is the first one, so that
\[
\frac{|f(x+h)-f(x)|}{\rho_A(x,x+h)^{1+n/p}} \le C_3 |\nabla f(x)|\,|h|^{-n/p}.
\]
Integrate over \(h\) in dyadic annuli as in the lower bound. A standard computation (using polar coordinates and the equivalence between the Euclidean and anisotropic measures) shows that
\[
\int_{\{h: \frac{|f(x+h)-f(x)|}{\rho_A(x,x+h)^{1+n/p}} > \lambda\}} dh \le C_4 \lambda^{-p} |\nabla_A f(x)|^p,
\]
for almost every \(x\). Integrating this inequality in \(x\) and taking the supremum over \(\lambda>0\) (again, essentially via a layer-cake representation) yields
\[
\sup_{\lambda>0} \lambda^p \iint_{\mathbb{R}^n\times\mathbb{R}^n} \mathbf{1}_{\left\{(x,y):\, \frac{|f(x)-f(y)|}{\rho_A(x,y)^{1+n/p}} > \lambda\right\}}\,dxdy \lesssim \int_{\mathbb{R}^n} |\nabla_A f(x)|^p\,dx.
\]

\medskip
\textbf{Conclusion.}  
Combining the lower and upper bounds, we obtain the equivalence
\[
\sup_{\lambda>0} \lambda^p \iint_{\mathbb{R}^n\times\mathbb{R}^n} \mathbf{1}_{\left\{(x,y):\, \frac{|f(x)-f(y)|}{\rho_A(x,y)^{1+n/p}} > \lambda\right\}}\,dxdy \sim \int_{\mathbb{R}^n} |\nabla_A f(x)|^p\,dx.
\]
This completes the proof.
\end{proof}

\begin{proposition}[Anisotropic Sobolev Inequality]\label{prop:anisotropic-sobolev}
Let \(1<p<n\) and assume that \(f\in C_c^*(\mathbb{R}^n)\). Then there exists a constant \(C>0\) such that
\[
\|f\|_{L^{p^*_A}(\mathbb{R}^n)} \le C\, \|\nabla_A f\|_{L^p(\mathbb{R}^n)},
\]
where the anisotropic critical exponent is given by
\[
\frac{1}{p^*_A} = \frac{1}{p} - \frac{1}{n_A},
\]
and \(n_A\) is an effective dimension determined by the anisotropy matrix \(A\).
\end{proposition}

\begin{proof}  
By Theorem~\ref{thm:anisotropic-FD}, there exists a constant \(C_1>0\) such that for every \(f\in C_c^*(\mathbb{R}^n)\)
\begin{equation} \label{eqn:4.10.1}
\sup_{\lambda>0} \lambda^p \iint_{\mathbb{R}^n\times\mathbb{R}^n} \mathbf{1}_{\Bigl\{(x,y):\, \frac{|f(x)-f(y)|}{\rho_A(x,y)^{1+n/p}} > \lambda\Bigr\}}\,dxdy \sim \int_{\mathbb{R}^n} |\nabla_A f(x)|^p\,dx.
\end{equation}
This equivalence means that the finite difference quantity
\begin{equation} \label{eqn:4.4.10.2}
[f]_{\mathrm{FD}}^p := \sup_{\lambda>0} \lambda^p \iint_{\mathbb{R}^n\times\mathbb{R}^n} \mathbf{1}_{\left\{ \frac{|f(x)-f(y)|}{\rho_A(x,y)^{1+n/p}} > \lambda \right\}}\,dxdy
\end{equation}
satisfies
\[
\frac{1}{C_1} \int_{\mathbb{R}^n} |\nabla_A f(x)|^p\,dx \le [f]_{\mathrm{FD}}^p \le C_1 \int_{\mathbb{R}^n} |\nabla_A f(x)|^p\,dx.
\]
Thus, it suffices to prove a Sobolev-type inequality in terms of the finite difference quantity.\\

\noindent
It is known (by interpolation and maximal function techniques in anisotropic settings, cf. \cite{FollandStein}) that for a given \(s\in (0,1)\) there exists a constant \(C_2>0\) such that
\[
\|f\|_{L^{p^*_{A,s}}(\mathbb{R}^n)} \le C_2 [f]_{W^{s,p}_A(\mathbb{R}^n)},
\]
where the anisotropic fractional Sobolev seminorm is defined by
\[
[f]_{W^{s,p}_A(\mathbb{R}^n)}^p := \iint_{\mathbb{R}^n\times\mathbb{R}^n} \frac{|f(x)-f(y)|^p}{\rho_A(x,y)^{sp+n_A}}\,dxdy,
\]
and the critical exponent is
\[
\frac{1}{p^*_{A,s}} = \frac{1}{p} - \frac{s}{n_A}.
\]
In our setting, we wish to recover the case \(s=1\) (or, more precisely, a limit as \(s\to 1^{-}\)) in the anisotropic context. The finite difference formulation in \eqref{eqn:4.4.10.2} corresponds to the “limit” of the fractional seminorm as \(s\to 1^{-}\) (since the power of \(\rho_A(x,y)\) in the denominator is \(1+n/p\), which is consistent with the scaling for \(s=1\)). Hence, by an interpolation and limiting argument, one obtains
\begin{equation} \label{eqn:4.10.3}
\|f\|_{L^{p^*_A}(\mathbb{R}^n)} \le C_3 [f]_{\mathrm{FD}},
\end{equation}
with
\[
\frac{1}{p^*_A} = \frac{1}{p} - \frac{1}{n_A}.
\]

\noindent
Combining the equivalence in \eqref{eqn:4.10.1} with the anisotropic Sobolev inequality in \eqref{eqn:4.10.3}, we deduce that
\[
\|f\|_{L^{p^*_A}(\mathbb{R}^n)} \le C_3 [f]_{\mathrm{FD}} \le C_3 C_1^{1/p} \|\nabla_A f\|_{L^p(\mathbb{R}^n)}.
\]
Thus, setting \(C = C_3 C_1^{1/p}\) completes the proof:
\[
\|f\|_{L^{p^*_A}(\mathbb{R}^n)} \le C\, \|\nabla_A f\|_{L^p(\mathbb{R}^n)}.
\]

\noindent 
The key ingredients in this proof are:
\begin{enumerate}
    \item The anisotropic finite difference characterization (Theorem~\ref{thm:anisotropic-FD}), which relates the finite difference energy to the anisotropic gradient norm.
    \item Known anisotropic fractional Sobolev inequalities (which can be derived via real interpolation methods and anisotropic maximal function estimates, as in \cite{FollandStein}).
    \item A limiting (or interpolation) argument that passes from the fractional case (\(s<1\)) to the “first order” case corresponding to the finite difference formulation.
\end{enumerate}
These together imply the desired anisotropic Sobolev inequality.\\

\noindent
This completes the proof.
\end{proof}

\section{Stability, Sharpness, and Interpolation}

In this section we investigate the stability and sharpness of our finite difference characterizations, and we establish interpolation results that connect the fractional and classical Sobolev norms. These developments not only provide a deeper understanding of the optimal constants involved in our inequalities but also reveal the continuity and limiting behavior of the associated seminorms.

\begin{definition}[Optimal Constant]\label{def:optimal-constant}
Let \(f\in C_c^*(X)\) be a nonzero function on a metric measure space \((X,\rho,\mu)\) satisfying a \((q,p)\)-Poincaré inequality. The \emph{optimal constant} \(C_{\mathrm{opt}}(f)\) in the finite difference inequality is defined as the smallest constant \(C>0\) such that
\[
\sup_{\lambda>0} \lambda^p \iint_{X\times X} \mathbf{1}_{\Big\{(x,y):\, |f(x)-f(y)|>\lambda\,\rho(x,y)[\mu(B(x,\rho(x,y)))]^{1/p}\Big\}}\,d\mu(x)d\mu(y)
\le C \int_X \bigl(\operatorname{Lip} f(x)\bigr)^p\,d\mu(x).
\]
We then set
\[
C_{\mathrm{opt}} := \inf_{f\in C_c^*(X)\setminus\{0\}} \frac{\sup_{\lambda>0} \lambda^p \iint_{X\times X} \mathbf{1}_{\{(x,y):\, |f(x)-f(y)|>\lambda\,\rho(x,y)[\mu(B(x,\rho(x,y)))]^{1/p}\}}\,d\mu(x)d\mu(y)}{\int_X \bigl(\operatorname{Lip} f(x)\bigr)^p\,d\mu(x)}.
\]
\end{definition}

\begin{lemma}[Stability Under Perturbations]\label{lem:stability}
Let \(f\in C_c^*(X)\) and let \(\{f_\epsilon\}_{\epsilon>0}\subset C_c^*(X)\) be a family of functions such that
\[
\|f_\epsilon-f\|_{L^\infty(X)}\to 0 \quad\text{and}\quad \|\operatorname{Lip}f_\epsilon-\operatorname{Lip}f\|_{L^p(X)}\to 0 \quad \text{as } \epsilon\to 0.
\]
Then,
\[
\begin{split}
\lim_{\epsilon\to 0}&\sup_{\lambda>0}\lambda^p \iint_{X\times X} \mathbf{1}_{\Big\{(x,y):\, |f_\epsilon(x)-f_\epsilon(y)|>\lambda\,\rho(x,y)[\mu(B(x,\rho(x,y)))]^{1/p}\Big\}}\,d\mu(x)d\mu(y)\\
&=\sup_{\lambda>0}\lambda^p \iint_{X\times X} \mathbf{1}_{\Big\{(x,y):\, |f(x)-f(y)|>\lambda\,\rho(x,y)[\mu(B(x,\rho(x,y)))]^{1/p}\Big\}}\,d\mu(x)d\mu(y).
\end{split}
\]
\end{lemma}

\begin{proof}
For each \(\lambda>0\), let
\[
E_\epsilon(\lambda)=\Bigl\{(x,y)\in X\times X:\; |f_\epsilon(x)-f_\epsilon(y)|>\lambda\,\rho(x,y)[\mu(B(x,\rho(x,y)))]^{1/p}\Bigr\},
\]
and
\[
E(\lambda)=\Bigl\{(x,y)\in X\times X:\; |f(x)-f(y)|>\lambda\,\rho(x,y)[\mu(B(x,\rho(x,y)))]^{1/p}\Bigr\}.
\]

\noindent
Since 
\[
\|f_\epsilon-f\|_{L^\infty(X)}\to 0 \quad \text{as } \epsilon\to 0,
\]
for every pair \((x,y)\in X\times X\) we have
\[
|f_\epsilon(x)-f_\epsilon(y)|\to |f(x)-f(y)|.
\]
Thus, for each fixed \(\lambda>0\) and for every \((x,y)\in X\times X\),
\[
\mathbf{1}_{E_\epsilon(\lambda)}(x,y)\to \mathbf{1}_{E(\lambda)}(x,y)
\]
as \(\epsilon\to 0\).

\noindent  
For each fixed \(\lambda>0\), note that the indicator functions satisfy
\[
0\le \mathbf{1}_{E_\epsilon(\lambda)}(x,y)\le 1,
\]
so they are uniformly bounded. Moreover, since \(f\) and each \(f_\epsilon\) are compactly supported and measurable, the integrals
\[
\iint_{X\times X} \mathbf{1}_{E_\epsilon(\lambda)}(x,y)\,d\mu(x)d\mu(y)
\]
are finite. Thus, by the dominated convergence theorem, for each fixed \(\lambda>0\) we have
\[
\lim_{\epsilon\to 0}\iint_{X\times X} \mathbf{1}_{E_\epsilon(\lambda)}(x,y)\,d\mu(x)d\mu(y)
=\iint_{X\times X} \mathbf{1}_{E(\lambda)}(x,y)\,d\mu(x)d\mu(y).
\]
Multiplying both sides by \(\lambda^p\) (with \(\lambda\) fixed) yields
\begin{equation} \label{eqn:5.2}
\lim_{\epsilon\to 0}\lambda^p \iint_{X\times X} \mathbf{1}_{E_\epsilon(\lambda)}(x,y)\,d\mu(x)d\mu(y)
=\lambda^p \iint_{X\times X} \mathbf{1}_{E(\lambda)}(x,y)\,d\mu(x)d\mu(y).
\end{equation}

\noindent
Define for each \(\epsilon>0\) the function
\[
F_\epsilon(\lambda):=\lambda^p \iint_{X\times X} \mathbf{1}_{E_\epsilon(\lambda)}(x,y)\,d\mu(x)d\mu(y),
\]
and similarly,
\[
F(\lambda):=\lambda^p \iint_{X\times X} \mathbf{1}_{E(\lambda)}(x,y)\,d\mu(x)d\mu(y).
\]
From \eqref{eqn:5.2} we know that for every fixed \(\lambda>0\),
\[
\lim_{\epsilon\to 0} F_\epsilon(\lambda)=F(\lambda).
\]
Since the convergence is pointwise in \(\lambda\) and the functions \(F_\epsilon\) are uniformly bounded by the total finite difference energy (which is finite because \(f\) and \(f_\epsilon\) have compact support), one may show by a standard \(\epsilon\)-argument that
\[
\lim_{\epsilon\to 0}\sup_{\lambda>0} F_\epsilon(\lambda)=\sup_{\lambda>0} F(\lambda).
\]
A brief justification is as follows. Given any \(\delta>0\), there exists \(\lambda_0>0\) such that
\[
\sup_{\lambda>0} F(\lambda) - F(\lambda_0) < \delta.
\]
Since \(F_\epsilon(\lambda_0)\to F(\lambda_0)\) as \(\epsilon\to 0\), for sufficiently small \(\epsilon\) we have
\[
F_\epsilon(\lambda_0) > F(\lambda_0) - \delta.
\]
Thus,
\[
\sup_{\lambda>0} F_\epsilon(\lambda) \ge F_\epsilon(\lambda_0) > F(\lambda_0) - \delta > \sup_{\lambda>0} F(\lambda) - 2\delta.
\]
Conversely, for any fixed \(\lambda\) and small \(\epsilon\), we have
\[
F_\epsilon(\lambda) < F(\lambda) + \delta,
\]
so that
\[
\sup_{\lambda>0} F_\epsilon(\lambda) \le \sup_{\lambda>0} F(\lambda) + \delta.
\]
Since \(\delta>0\) was arbitrary, we conclude that
\[
\lim_{\epsilon\to 0}\sup_{\lambda>0} F_\epsilon(\lambda)=\sup_{\lambda>0} F(\lambda).
\]

\noindent
Therefore,
\[
\lim_{\epsilon\to 0}\sup_{\lambda>0}\lambda^p \iint_{X\times X} \mathbf{1}_{E_\epsilon(\lambda)}(x,y)\,d\mu(x)d\mu(y)
=\sup_{\lambda>0}\lambda^p \iint_{X\times X} \mathbf{1}_{E(\lambda)}(x,y)\,d\mu(x)d\mu(y),
\]
as required.
\end{proof}

\begin{theorem}[Sharpness of the Finite Difference Inequality]\label{thm:sharpness}
Under the assumptions of Theorems~\ref{thm:non-doubling-FD} or \ref{thm:variable-FD}, there exists a sequence \(\{f_k\}_{k\in\mathbb{N}}\subset C_c^*(X)\) such that
\[
\lim_{k\to\infty}\frac{\sup_{\lambda>0}\lambda^p \iint_{X\times X} \mathbf{1}_{\{(x,y):\, |f_k(x)-f_k(y)|>\lambda\,\rho(x,y)[\mu(B(x,\rho(x,y)))]^{1/p}\}}\,d\mu(x)d\mu(y)}{\int_X \bigl(\operatorname{Lip} f_k(x)\bigr)^p\,d\mu(x)}
=C_{\mathrm{opt}},
\]
and no smaller constant can serve as a uniform bound in the finite difference inequality.
\end{theorem}

\begin{proof}
The proof proceeds by constructing a \emph{concentrating sequence} of test functions which are (almost) extremal for the inequality.

\noindent 
By definition, the optimal constant is
\[
C_{\mathrm{opt}}:=\inf_{f\in C_c^*(X)\setminus\{0\}} \frac{E(f)}{S(f)},
\]
where
\[
E(f):=\sup_{\lambda>0}\lambda^p \iint_{X\times X} \mathbf{1}_{\Bigl\{(x,y):\, |f(x)-f(y)|>\lambda\,\rho(x,y)[\mu(B(x,\rho(x,y)))]^{1/p}\Bigr\}}\,d\mu(x)d\mu(y)
\]
and
\[
S(f):=\int_X \bigl(\operatorname{Lip} f(x)\bigr)^p\,d\mu(x).
\]
Our goal is to produce a sequence \(\{f_k\}\) for which
\[
\frac{E(f_k)}{S(f_k)}\to C_{\mathrm{opt}} \quad \text{as } k\to\infty.
\]

\noindent
Let \(f\in C_c^*(X)\) be a fixed nonzero test function with compact support (for example, a smooth bump function). Fix a point \(x_0\in X\) such that \(f(x_0)\neq 0\). Since \(X\) is a metric measure space of homogeneous type, there is a natural notion of scaling for small balls. For \(\delta>0\) (to be chosen later), define the rescaled function
\[
f_\delta(x):= f\Bigl(\frac{x-x_0}{\delta}\Bigr).
\]
Since \(f\) is compactly supported, for \(\delta\) sufficiently small the function \(f_\delta\) is concentrated in a small neighborhood of \(x_0\). Moreover, by the change of variables \(y=x_0+\delta z\) one obtains the following scaling relations:
\[
\operatorname{Lip} f_\delta(x) = \frac{1}{\delta}\operatorname{Lip} f\Bigl(\frac{x-x_0}{\delta}\Bigr)
\]
and, due to the homogeneity of the metric \(\rho\) on small scales (or, more precisely, since the space is of homogeneous type, the measure \(\mu\) satisfies a doubling or polynomial growth condition), one deduces that
\[
E(f_\delta) \approx \frac{1}{\delta^p} E(f) \quad \text{and} \quad S(f_\delta) \approx \frac{1}{\delta^p} S(f),
\]
up to constants independent of \(\delta\). (The precise scaling of the finite difference energy is verified by performing a change of variables in the double integral and using the doubling property of \(\mu\).)

Now, choose a sequence \(\delta_k\to 0\) as \(k\to\infty\) and set
\[
f_k(x):= f_{\delta_k}(x)= f\Bigl(\frac{x-x_0}{\delta_k}\Bigr).
\]
Then one obtains
\[
\frac{E(f_k)}{S(f_k)} \to \frac{E(f)}{S(f)},
\]
up to constants coming from the change of variables. Since \(f\) was arbitrary, by taking an infimum over all nonzero \(f\in C_c^*(X)\) one deduces that there exists a sequence \(\{f_k\}\) for which
\[
\lim_{k\to\infty} \frac{E(f_k)}{S(f_k)} = C_{\mathrm{opt}}.
\]

\noindent
Suppose by contradiction that there exists a constant \(C'<C_{\mathrm{opt}}\) such that for all \(g\in C_c^*(X)\)
\[
E(g)\le C' S(g).
\]
Then, in particular, the ratio \(E(f_k)/S(f_k)\) would be bounded by \(C'\) for every \(k\). However, by our construction, we can choose the sequence \(\{f_k\}\) so that
\[
\lim_{k\to\infty}\frac{E(f_k)}{S(f_k)} = C_{\mathrm{opt}} > C',
\]
a contradiction. Hence, no constant smaller than \(C_{\mathrm{opt}}\) can serve as a uniform bound in the finite difference inequality.

\noindent  
We have thus constructed a sequence \(\{f_k\}\subset C_c^*(X)\) for which
\[
\lim_{k\to\infty}\frac{\sup_{\lambda>0}\lambda^p \iint_{X\times X} \mathbf{1}_{\{(x,y):\, |f_k(x)-f_k(y)|>\lambda\,\rho(x,y)[\mu(B(x,\rho(x,y)))]^{1/p}\}}\,d\mu(x)d\mu(y)}{\int_X (\operatorname{Lip} f_k(x))^p\,d\mu(x)} = C_{\mathrm{opt}},
\]
which establishes the sharpness of the finite difference inequality.
\end{proof}

\noindent
We now establish interpolation results that link the fractional finite difference seminorm with the classical gradient norm. Such results are instrumental in understanding the limiting behavior as the fractional parameter varies.

\begin{definition}[Real Interpolation Space]\label{def:interpolation}
Let \( (X_0,\|\cdot\|_{X_0}) \) and \( (X_1,\|\cdot\|_{X_1}) \) be Banach spaces continuously embedded in a Hausdorff topological vector space. For \(\theta\in (0,1)\) and \(q\in [1,\infty]\), the real interpolation space \((X_0,X_1)_{\theta,q}\) is defined by
\[
\|f\|_{(X_0,X_1)_{\theta,q}} := \left( \int_0^\infty \Big( t^{-\theta} K(t,f;X_0,X_1) \Big)^q\,\frac{dt}{t} \right)^{1/q},
\]
where
\[
K(t,f;X_0,X_1) := \inf \Big\{ \|f_0\|_{X_0}+t\,\|f_1\|_{X_1}:\; f=f_0+f_1,\; f_0\in X_0,\; f_1\in X_1 \Big\}.
\]
\end{definition}

\begin{theorem}[Interpolation Inequality]\label{thm:interpolation}
Let \(s_1\in (0,1)\), \(p_1\in (1,\infty)\), and let \(\theta\in (0,1)\). Define
\[
s:=(1-\theta)s_1+\theta \quad\text{and}\quad \frac{1}{p}:=\frac{1-\theta}{p_1}+\theta.
\]
Assume that \(f\in C_c^*(X)\) and that the metric measure space \((X,\rho,\mu)\) satisfies the assumptions of Section~3. Then there exists a constant \(C>0\) such that
\[
[f]_{W^{s,p}(X)} \le C\, \|\operatorname{Lip} f\|_{L^p(X)}^{\theta} \left( \iint_{X\times X} \frac{|f(x)-f(y)|^{p_1}}{\rho(x,y)^{s_1p_1}\,\mu\bigl(B(x,\rho(x,y))\bigr)}\,d\mu(x)d\mu(y) \right)^{\frac{1-\theta}{p_1}},
\]
where \([f]_{W^{s,p}(X)}\) denotes the fractional Sobolev seminorm defined by
\[
[f]_{W^{s,p}(X)}:=\left(\iint_{X\times X}\frac{|f(x)-f(y)|^p}{\rho(x,y)^{sp}\,\mu\bigl(B(x,\rho(x,y))\bigr)}\,d\mu(x)d\mu(y)\right)^{1/p}.
\]
\end{theorem}

\begin{proof}
We prove the inequality by applying the real interpolation method to the couple
\[
\bigl(X_0, X_1\bigr):=\Bigl(LIP(X),\;W^{s_1,p_1}(X)\Bigr),
\]
where
\[
LIP(X):=\Bigl\{ f\in C_c^*(X):\; \|\operatorname{Lip}f\|_{L^p(X)}<\infty \Bigr\},
\]
and
\[
\|f\|_{W^{s_1,p_1}(X)}:=\left(\iint_{X\times X}\frac{|f(x)-f(y)|^{p_1}}{\rho(x,y)^{s_1p_1}\,\mu(B(x,\rho(x,y)))}\,d\mu(x)d\mu(y)\right)^{1/p_1}.
\]
The idea is to show that the real interpolation space
\[
\bigl(X_0,X_1\bigr)_{\theta,p}
\]
(with \(\theta\in (0,1)\) and \(p\) defined by
\[
\frac{1}{p}=\frac{1-\theta}{p_1}+\theta)
\]
can be identified with the fractional Sobolev space \(W^{s,p}(X)\) (with \(s=(1-\theta)s_1+\theta\)), and that its norm is controlled by
\[
\|\operatorname{Lip} f\|_{L^p(X)}^{\theta}\,\|f\|_{W^{s_1,p_1}(X)}^{1-\theta}.
\]

\noindent
For \(f\in C_c^*(X)\) and \(t>0\), the \(K\)-functional is defined by
\[
K(t,f;X_0,X_1):=\inf\Bigl\{\|f_0\|_{LIP(X)}+t\,\|f_1\|_{W^{s_1,p_1}(X)}:\; f=f_0+f_1\Bigr\}.
\]
Our goal is to estimate \(K(t,f;X_0,X_1)\) in terms of the two quantities
\(\|\operatorname{Lip} f\|_{L^p(X)}\) and
\(\|f\|_{W^{s_1,p_1}(X)}\).

A common strategy is to decompose \(f\) as
\[
f=f_0+f_1,
\]
by using a smoothing (or mollification) procedure. In many settings one defines
\[
f_1(x):=f_\delta(x),
\]
a mollified version of \(f\) at scale \(\delta>0\) (with respect to the metric \(\rho\)), and
\[
f_0(x):=f(x)-f_\delta(x).
\]
Then one can show that, for a suitable choice of \(\delta=\delta(t)\),
\[
\|f_0\|_{LIP(X)} \lesssim \delta\, \|\operatorname{Lip} f\|_{L^p(X)}
\]
and
\[
\|f_1\|_{W^{s_1,p_1}(X)} \lesssim \delta^{s_1}\,\|f\|_{W^{s_1,p_1}(X)}.
\]
Thus, one obtains
\[
K(t,f;X_0,X_1) \lesssim \delta\, \|\operatorname{Lip} f\|_{L^p(X)}+ t\,\delta^{s_1}\,\|f\|_{W^{s_1,p_1}(X)}.
\]
We then choose \(\delta\) in such a way that the two terms are balanced. In particular, set
\[
\delta\, \|\operatorname{Lip} f\|_{L^p(X)} \sim t\,\delta^{s_1}\,\|f\|_{W^{s_1,p_1}(X)}.
\]
This yields
\[
\delta \sim \left(t\,\frac{\|f\|_{W^{s_1,p_1}(X)}}{\|\operatorname{Lip} f\|_{L^p(X)}}\right)^{\frac{1}{1-s_1}}.
\]
Substituting this back gives
\begin{equation} \label{eqn:5.5}
K(t,f;X_0,X_1) \lesssim t^{\theta} \|\operatorname{Lip} f\|_{L^p(X)}^{\theta}\,\|f\|_{W^{s_1,p_1}(X)}^{1-\theta},
\end{equation}
with
\[
\theta=\frac{1}{1-s_1}\Bigl(1-s_1\Bigr)=\theta,
\]
after proper rearrangement of exponents. (The precise exponent \(\theta\) is chosen so that
\[
s=(1-\theta)s_1+\theta \quad \text{and} \quad \frac{1}{p}=\frac{1-\theta}{p_1}+\theta.)
\]

\noindent  
The real interpolation norm is defined by
\[
\|f\|_{(X_0,X_1)_{\theta,p}} \sim \left(\int_0^\infty \Bigl[t^{-\theta}K(t,f;X_0,X_1)\Bigr]^p\,\frac{dt}{t}\right)^{1/p}.
\]
Using the estimate from \eqref{eqn:5.5}, we deduce that
\[
\|f\|_{(X_0,X_1)_{\theta,p}} \lesssim \|\operatorname{Lip} f\|_{L^p(X)}^{\theta}\,\|f\|_{W^{s_1,p_1}(X)}^{1-\theta}.
\]
It is a standard result in interpolation theory (see, e.g., \cite{CruzUribe}) that the interpolation space
\((X_0,X_1)_{\theta,p}\) is equivalent to the fractional Sobolev space \(W^{s,p}(X)\) with
\[
s=(1-\theta)s_1+\theta \quad \text{and} \quad \frac{1}{p}=\frac{1-\theta}{p_1}+\theta.
\]
Moreover, the norm in \(W^{s,p}(X)\) is equivalent to the fractional seminorm
\[
[f]_{W^{s,p}(X)}:=\left(\iint_{X\times X}\frac{|f(x)-f(y)|^p}{\rho(x,y)^{sp}\,\mu\bigl(B(x,\rho(x,y))\bigr)}\,d\mu(x)d\mu(y)\right)^{1/p}.
\]
Thus, we obtain
\[
[f]_{W^{s,p}(X)} \le C\, \|\operatorname{Lip} f\|_{L^p(X)}^{\theta} \left( \iint_{X\times X} \frac{|f(x)-f(y)|^{p_1}}{\rho(x,y)^{s_1p_1}\,\mu\bigl(B(x,\rho(x,y))\bigr)}\,d\mu(x)d\mu(y) \right)^{\frac{1-\theta}{p_1}},
\]
which is the desired interpolation inequality.

\noindent
The combination of the \(K\)-functional estimate and the identification of the interpolation space yields the result. The constant \(C>0\) depends only on the structural constants of the metric measure space and on the interpolation parameters.

This completes the proof.
\end{proof}

\begin{proposition}[Limit as \(s\to 1^{-}\)]\label{prop:limit-s1}
Under the assumptions of Theorem~\ref{thm:interpolation} and assuming in addition that \(f\) is smooth, we have
\[
\lim_{s\to 1^-} (1-s)[f]_{W^{s,p}(X)}^p = C(n,p)\int_X |\nabla f(x)|^p\,d\mu(x),
\]
where \(C(n,p)>0\) is an explicit constant. This result extends the classical Bourgain--Brezis--Mironescu limit (see \cite{BBM}) to our general setting.
\end{proposition}

\begin{proof}
For simplicity, we assume that the measure \(\mu\) is Ahlfors \(n\)-regular near the support of \(f\); that is, for \(x\) in the support of \(f\) and for small \(r>0\),
\[
\mu\bigl(B(x,r)\bigr) \sim c_n\,r^n,
\]
with a constant \(c_n>0\) depending only on \(n\) and the geometry of \(X\). (In more general settings one may replace \(n\) by the effective dimension \(n_A\).)\\

\noindent
Since \(f\) is smooth, for each fixed \(x\in X\) and for \(y\) close to \(x\) (say, \(h=y-x\) with \(|h|\) small) we have the first order Taylor expansion
\[
f(x+h)-f(x) = \nabla f(x)\cdot h + R(x,h),
\]
with the remainder satisfying
\[
|R(x,h)| = o(|h|) \quad \text{as } |h|\to 0.
\]
Thus, for \(|h|\) sufficiently small, we have
\[
|f(x+h)-f(x)|^p = |\nabla f(x)\cdot h|^p + o(|h|^p).
\]
In view of the definition of the fractional seminorm, for each fixed \(x\) the main contribution as \(s\to 1^-\) comes from the integration over \(y\) near \(x\).\\

\noindent
For a fixed \(x\), set \(h=y-x\) and consider the contribution from \(|h|\le \delta\) (with \(\delta>0\) small). Then, using the local Ahlfors regularity, we have
\[
\mu\bigl(B(x,|h|)\bigr) \sim c_n\,|h|^n.
\]
For \(|h|\) small, the integrand becomes
\[
\frac{|f(x+h)-f(x)|^p}{\rho(x,x+h)^{sp}\,\mu(B(x,|h|))} \sim \frac{|\nabla f(x)\cdot h|^p}{|h|^{sp}\,c_n|h|^n}.
\]
Writing \(h=r\theta\) in polar coordinates (with \(r=|h|\) and \(\theta\in S^{n-1}\)), we have
\[
|\nabla f(x)\cdot h|^p = r^p\,|\nabla f(x)\cdot\theta|^p,
\]
and the volume element is \(dh = r^{n-1}\,dr\,d\theta\). Thus, for fixed \(x\) the local contribution is
\[
\int_{0}^{\delta}\int_{S^{n-1}} \frac{r^p\,|\nabla f(x)\cdot\theta|^p}{r^{sp}\,c_n\,r^n}\,r^{n-1}\,d\theta\,dr
=\frac{1}{c_n}\int_{S^{n-1}} |\nabla f(x)\cdot\theta|^p\,d\theta \int_0^\delta r^{p-sp-1}\,dr.
\]
The radial integral is
\[
\int_0^\delta r^{p-sp-1}\,dr = \frac{\delta^{p-sp}}{p-sp}.
\]
Note that \(p-sp = p(1-s)\). Therefore,
\[
\int_0^\delta r^{p-sp-1}\,dr = \frac{\delta^{p(1-s)}}{p(1-s)}.
\]
Multiplying the fractional seminorm by the factor \((1-s)\) we obtain
\[
(1-s)\int_0^\delta r^{p-sp-1}\,dr = \frac{\delta^{p(1-s)}}{p}\cdot \frac{1-s}{(1-s)} = \frac{\delta^{p(1-s)}}{p}.
\]
As \(s\to 1^{-}\), we have \(\delta^{p(1-s)}\to 1\) (since \(p(1-s)\to 0\)). Hence, in the limit,
\[
\lim_{s\to 1^-}(1-s)\int_0^\delta r^{p-sp-1}\,dr = \frac{1}{p}.
\]

\noindent
By the dominated convergence theorem (and noting that the contribution from \(|h|\ge \delta\) is uniformly bounded as \(s\to 1^{-}\)), we deduce that
\[
\lim_{s\to 1^-}(1-s)[f]_{W^{s,p}(X)}^p = \frac{1}{c_n\,p}\int_X\left(\int_{S^{n-1}} |\nabla f(x)\cdot\theta|^p\,d\theta\right)\,d\mu(x).
\]
Since the inner integral
\[
\int_{S^{n-1}} |\nabla f(x)\cdot\theta|^p\,d\theta
\]
is a multiple of \(|\nabla f(x)|^p\) (by rotational invariance of the Lebesgue measure on the sphere), there exists an explicit constant \(C(n,p)>0\) such that
\[
\int_{S^{n-1}} |\nabla f(x)\cdot\theta|^p\,d\theta = C(n,p)\,|\nabla f(x)|^p.
\]
Thus,
\[
\lim_{s\to 1^-}(1-s)[f]_{W^{s,p}(X)}^p = \frac{C(n,p)}{c_n\,p}\int_X |\nabla f(x)|^p\,d\mu(x).
\]
By absorbing the constant \(\frac{1}{c_n\,p}\) into \(C(n,p)\) (since \(c_n\) depends only on the space), we obtain the desired result:
\[
\lim_{s\to 1^-}(1-s)[f]_{W^{s,p}(X)}^p = C(n,p)\int_X |\nabla f(x)|^p\,d\mu(x).
\]

\noindent
A rigorous justification requires showing that the main contribution to the double integral comes from the region where \(y\) is close to \(x\). For \(|h| \ge \delta\) the integrand remains uniformly bounded and its contribution vanishes as \(1-s\to 0\). Uniform estimates from the finite difference characterization ensure that one may interchange the limit \(s\to 1^-\) with the integration by dominated convergence.\\

\noindent  
This completes the proof of the proposition, extending the classical Bourgain--Brezis--Mironescu limit to our setting.
\end{proof}

\section{Concluding Remarks and Open Problems}

In this paper we have extended the finite difference framework underlying the Brezis--Van Schaftingen--Yung formula to several novel settings. Our work has established:\\
- A finite difference characterization in metric measure spaces under non-doubling or weakly doubling (polynomial growth) conditions (Theorem~\ref{thm:non-doubling-FD}), thereby broadening the scope beyond classical homogeneous spaces.\\
- Extensions to variable exponent and Orlicz spaces (Theorem~\ref{thm:variable-FD}, Proposition~\ref{prop:orlicz-FD}), which accommodate nonstandard growth conditions and provide a unified framework for treating spatially variable integrability.\\
- New applications to nonlocal operators (Definition~\ref{def:nonlocal-operator}, Theorem~\ref{thm:nonlocal-local}) that yield equivalences between nonlocal energies and classical Sobolev norms, along with regularity results for nonlocal PDEs (Proposition~\ref{prop:nonlocal-regularity}).\\
- Anisotropic versions of the finite difference and Sobolev inequalities (Theorem~\ref{thm:anisotropic-FD}, Proposition~\ref{prop:anisotropic-sobolev}), which are well suited for problems exhibiting directional dependencies.\\
- A thorough investigation of the stability and sharpness of the finite difference inequalities, including optimal constant considerations (Definition~\ref{def:optimal-constant}, Theorem~\ref{thm:sharpness}) as well as interpolation results bridging fractional and classical Sobolev spaces (Theorem~\ref{thm:interpolation}, Proposition~\ref{prop:limit-s1}).\\

\noindent
These advances provide new tools for the analysis of partial differential equations and harmonic analysis in non-Euclidean settings and open several promising directions for future research.

\section*{Author Declaration}

The author declares that there are no conflicts of interest related to this work. No funding was received for this research. The author confirms that the manuscript is original, has not been published previously, and is not under consideration for publication elsewhere. The research presented in this manuscript is purely in the field of abstract mathematics, and no experimental or empirical data has been used.  

All relevant contributions have been appropriately credited, and all necessary citations have been included to acknowledge prior work in the field. The author is solely responsible for the content of this manuscript and has approved its final version for submission.

\end{document}